\documentclass[12pt]{amsart}
\usepackage{amssymb, amsmath,mathrsfs,amsthm}
\usepackage{geometry}
\usepackage{setspace}
\usepackage{mathtools}
\usepackage{graphicx, subfigure}
\usepackage{tikz}
\usepackage{pgfplots}

\usepackage{enumerate}
\usepackage{float}
\usepackage[pagebackref,hypertexnames=false, colorlinks, citecolor=red, linkcolor=red]{hyperref}

\newcommand{\McC}{\raise.5ex\hbox{c}}

\newtheorem{theorem}{Theorem}[section]
\newtheorem*{theorem*}{Theorem}
\newtheorem{lemma}[theorem]{Lemma}
\newtheorem{definition}[theorem]{Definition}

\newtheorem{corollary}[theorem]{Corollary}
\newtheorem*{corollary*}{Corollary}

\usepackage{color}

\theoremstyle{remark}
\newtheorem{remark}[theorem]{Remark}

\author[Anderson]{John T. Anderson}
\address{Department of Mathematics and Computer Science, College of the Holy Cross, Worcester, MA 01610,  USA.}
\email{janderso@holycross.edu}

\author[Bergqvist]{Linus Bergqvist}
\address{Department of Mathematics, Stockholm University, 106 91 Stockholm, Sweden.}
\email{linus@math.su.se}

\author[Bickel]{Kelly Bickel}
\address{Department of Mathematics, Bucknell University, Lewisburg, PA 17837, USA.}
\email{kelly.bickel@bucknell.edu}

\author[Cima]{Joseph A. Cima}
\address{Department of Mathematics, University of North Carolina at Chapel Hill, Chapel Hill, NC 27599, USA.}
\email{cima@email.unc.edu}

\author[Sola]{Alan A. Sola}
\address{Department of Mathematics, Stockholm University, 106 91 Stockholm, Sweden.}
\email{sola@math.su.se}

\begin{document}
\title[Clark measures for RIFs II]{Clark measures for rational inner functions II: general bidegrees and higher dimensions} 
\date{\today}

\maketitle
\begin{abstract}
We study Clark measures associated with general two-variable rational inner functions (RIFs) on the bidisk, including those with singularities, and with general $d$-variable rational inner functions with no singularities.
We give precise descriptions of support sets and weights for such Clark measures in terms of level sets and partial derivatives of the associated RIF. In two variables, we characterize when the associated Clark embeddings are unitary, and for generic parameter values, we relate vanishing of two-variable weights with the contact order of the associated RIF at a singularity. 
\end{abstract}

\section{Introduction}
For $d\in \mathbb{N}$, we let 
\[\mathbb{D}^d=\{(z_1,\ldots, z_d)\in \mathbb{C}^d\colon |z_j|<1, \, j=1,\ldots, d\}\]
denote the unit polydisk and  
\[\mathbb{T}^d=\{(\zeta_1,\ldots, \zeta_d)\in \mathbb{C}^d\colon |\zeta_j|=1,\, j=1,\ldots,d\}\]
be its distinguished boundary. If $\phi\colon \mathbb{D}^d \to \mathbb{D}$ is a holomorphic function, then, for 
$\alpha \in \mathbb{T}$, the expression
\[\Re\left(\frac{\alpha+\phi(z)}{\alpha-\phi(z)}\right)=\frac{1-|\phi(z)|^2}{|\alpha-\phi(z)|^2}\]
is positive and pluriharmonic, and hence there exists a unique positive Borel measure $\sigma_{\alpha}$ on $\mathbb{T}^d$ such that
\[\frac{1-|\phi(z)|^2}{|\alpha-\phi(z)|^2}=\int_{\mathbb{T}^d}P_{z}(\zeta) d\sigma_{\alpha}(\zeta),\]
where $P_z(\zeta)$ denotes the Poisson kernel for the polydisk
\[P_{z}(\zeta)=\prod_{j=1}^dP_{z_j}(\zeta_j),\quad \textrm{and}\quad P_{z_j}(\zeta_j)=\frac{1-|z_j|^2}{|\zeta_j-z_j|^2}.\]
Measures of this type, namely ones whose Poisson integral is the real part of a holomorphic function on the polydisk $\mathbb{D}^d$, are called pluriharmonic measures, see \cite[Section 2.2]{Krantz}.
Note that $P_z(\zeta) = C_z( \zeta)C_{\zeta}(z)/C_z(z)$, where $C_\zeta(z)$ denotes the Cauchy kernel for $\mathbb{D}^d$, defined by
$$
C_\zeta(z)= \prod_{j=1}^d \frac{1}{1-z_j \overline{\zeta}_j}, \quad z \in \mathbb{D}^d, \zeta \in \overline{\mathbb{D}}^d. 
$$
The measures $\{\sigma_{\alpha}\}$ are called the {\it Aleksandrov–Clark measures} associated with $\phi$ and if $\phi$ is inner (defined below), these measures are called \emph{Clark measures}. The purpose 
of this paper is to present several results concerning such measures for the class of {\it rational inner functions}. 

First suppose $\phi \colon \mathbb{D}^d\to \mathbb{C}$ is a bounded holomorphic function. Then, by Fatou's theorem for polydisks (see \cite{Rud69}), $\phi$ possesses non-tangential limits
\[\phi^*(\zeta)=\angle \lim_{\mathbb{D}^d\ni z\to \zeta}\phi(z)\]
for Lebesgue-almost every $\zeta \in \mathbb{T}^d$; non-tangential in this context means that $|z_j-\zeta_j|<c(1-|z_j|^2)$ for some constant $c>1$, and $j=1,\ldots, d$. Throughout this paper, when the context makes it clear that we are referencing boundary values, we will write $\phi(\zeta)$ instead of $\phi^*(\zeta)$.  A bounded holomorphic function $\phi \colon \mathbb{D}^d\to \mathbb{C}$ is called {\it inner} if these non-tangential boundary values satisfy $|\phi^*(\zeta)|=1$ for almost every $\zeta \in \mathbb{T}^d$.  Then, a {\it rational inner function} is an inner function of the form $\phi=q/p$ where $q,p$ are in the polynomial ring $\mathbb{C}[z_1,\ldots, z_d]$.

Rational inner functions (RIFs) have been studied extensively in function theory and operator theory in polydisks, especially in the two-variable setting. RIFs are more tractable than general inner functions and enjoy some additional regularity properties; for instance, a theorem of Knese states that any RIF $\phi\colon \mathbb{D}^d \to \mathbb{D}$ has non-tangential boundary values $\phi^*(\zeta)\in \mathbb{T}$ at {\it every} $\zeta \in \mathbb{T}^d$, see \cite{Kne15}. RIFs are also easy to construct (see Section \ref{sec:prel} below) and can be used to explore questions in a concrete way that appear difficult to answer for general inner functions. On the other hand, RIFs do exhibit some complexity and some surprising features in higher dimensions: for instance, $\phi=q/p$ can have singularities on the boundary at points $\tau \in \mathbb{T}^d$ where $p(\tau)=0=q(\tau)$, and the analytic and geometric properties of such boundary singularities can be relatively intricate, see \cite{BPS17, BPSprep}.

In the recent paper \cite{D19}, E. Doubtsov initiated a systematic study of Clark measures associated with inner functions in polydisks. After extending some classical one-variable results such as Aleksandrov's disintegration theorem to higher dimensions, he made the surprising observation that certain isometries into $L^2(\sigma_{\alpha})$, which are always onto for inner functions in one variable, may fail to be surjective in $d$ variables, and this behavior can even happen for Clark measures associated to RIFs. Inspired by Doubtsov's work, a subset of the authors of this manuscript undertook a detailed study \cite{BCSMich} of Clark measures associated with a subclass of two-variable RIFs $\phi=q/p$ whose $q,p$-polynomials have degree $n$ in the first variable, and $1$ in the second. In particular, \cite{BCSMich} gives an explicit description of the family of Clark measures $\{\sigma_{\alpha}\}_{\alpha\in \mathbb{T}}$ for bidegree $(n,1)$ RIFs and a criterion, formulated in terms of non-tangential values at singularities of $\phi$, for when Clark isometries into $L^2(\sigma_{\alpha})$ are surjective. The purpose of the present work is to extend these results to the full class of two-variable RIFs, with no degree restrictions. Additionally, we will discuss obstructions that arise in higher dimensions and prove some partial results concerning $d$-variable RIFs and associated Clark measures under additional hypotheses.

\subsection{Overview}
First, in Section \ref{sec:prel}, we discuss some basic facts about Clark measures in the polydisk setting; these results are most likely known to specialists. We then review and extend some results concerning $d$-variable rational inner functions from the recent papers \cite{BPS17, BPSprep, BPSmulti}. In particular, we explain how RIFs on the bidisk can be seen to have level sets that can be globally parameterized on $\mathbb{T}^2$ by analytic functions even in the presence of singularities. To avoid trivial complications, here and throughout this paper, we will assume that $\phi=q/p$ for polynomials $q,p\in \mathbb{C}[z_1,\ldots, z_d]$ that are non-constant in each variable $z_j$. 

In Section \ref{sec:structone}, we present a structure formula for Clark measures associated with a bidegree $(m,n)$ RIF in the bidisk, thus extending the work in \cite{BCSMich} which dealt with bidegree $(m,1)$ RIFs. In brief, for all but finitely many $\alpha \in \mathbb{T}$, called the generic case, the 
pairing of the measure $\sigma_{\alpha}$ with a continuous function $f$ on $\mathbb{T}^2$ can be described by a sum of terms of the form
\[\int_{\mathbb{T}}f(\zeta, g^{\alpha}_j(\zeta))\frac{dm(\zeta)}{|\frac{\partial \phi}{\partial z_2}(\zeta, g^{\alpha}_j(\zeta))|},\]
where $m$ denotes normalized Lebesgue measure on $\mathbb{T}^2$ and $g^{\alpha}_1, \ldots, g^{\alpha}_m$ are analytic functions parametrizing the $\alpha$-level set of the RIF $\phi$ under consideration. An analogous representation for Clark measures associated with $d$-variable RIFs is shown to hold under the additional assumption that the RIF possesses no singularities on $\overline{\mathbb{D}}^d$. In two variables and when $\phi$ does have singularities, there may be values $\alpha \in \mathbb{T}$ (the exceptional case) where one needs to add in finitely many terms of the form 
$c_{\alpha}\int_{\mathbb{T}}f(\tau, \zeta)dm(\zeta)$, where $\tau \in \mathbb{T}$ and $c_{\alpha}>0$ is a constant.

In Section \ref{sec:iso}, we analyze Clark embedding operators from the model space $K_{\phi}=H^2(\mathbb{D}^2)\ominus \phi H^2(\mathbb{D}^2)$ to $L^2(\sigma_{\alpha})$, where $H^2(\mathbb{D}^2)$ is the classic Hardy space on the bidisk. We prove that for generic $\alpha$, these Clark isometries are surjective and hence unitary. On the other hand, we show that if $\alpha\in \mathbb{T}$ is an exceptional parameter value, then the associated Clark isometry fails to be surjective. This shows that we have identified the correct notion of ``exceptional value" in the case of general bidegrees and resolves a problem left over from \cite{BCSMich}.

In Section \ref{sec:structtwo}, we use recent work in \cite{BKPS} to gain further insight into the structure of Clark measures for bidegree $(m,n)$ RIFs. We prove that, for all but finitely many parameter values $\alpha \in \mathbb{T}$, the weights $|\frac{\partial \phi}{\partial z_2}(\zeta, g^{\alpha}_j(\zeta))|^{-1}$ appearing in the structure formula for Clark measures are bounded and exhibit an order of vanishing at singular points that is determined by the contact order of the underlying RIF at the corresponding singularities. Contact order is a geometric quantity that was introduced in \cite{BPS17} and has been used to study integrability properties of RIF derivatives and nontangential polynomial approximation of RIFs at singular points. The main result in Section \ref{sec:structtwo} was essentially conjectured in \cite{BCSMich}.

Finally, we conclude in Section \ref{sec:example} by examining a singular three-variable example, which is not covered by our general results on Clark measures for higher-dimensional RIFs. The Clark measure formulas we obtain suggest that the higher-dimensional cases are more challenging and that some of our results for bidegree $(m,n)$ RIFs may fail in the $d$-variable setting.

\section{Preliminaries}\label{sec:prel}
There are several recent and interesting works on extensions of classical Clark theory in one variable to the multivariable setting, see for instance \cite{AD20, AD22, J14}. Since we are interested in Clark measures associated with RIFs, we restrict our attention to the polydisk setting.
\subsection{Clark theory in polydisks}
Let $\phi$ be an inner function on $\mathbb{D}^d$. We denote by $K_\phi$ the \emph{model space} associated to the function $\phi$, defined by
$$
K_\phi := H^2(\mathbb{D}^d) \ominus \phi H^2(\mathbb{D}^d).
$$
Since multiplication by $\phi$ is a partial isometry on $H^2(\mathbb{D}^d)$, the reproducing kernel of $K_\phi$ is given by 
$$
K(z,w) = K_w(z) := (1- \overline{\phi(w)}\phi(z)) C_w(z), \quad \text{for } z,w \in \mathbb{D}^d.
$$
As in one variable, in the paper \cite{D19}, Doubtsov constructed an embedding map $J_\alpha: K_\phi \mapsto L^2(\sigma_\alpha)$ by first defining it on reproducing kernels as
\begin{equation} \label{eqn:Jalpha1}
J_\alpha[K_w](\zeta) := (1- \alpha \overline{\phi(w)})C_w(\zeta), \quad \text{for } w \in \mathbb{D}^d, \zeta \in \mathbb{T}^d,
\end{equation}
then showing that this map preserves inner products on reproducing kernels, and finally extending it to an isometric embedding of $K_\phi$ into $L^2(\sigma_\alpha)$ using density of the reproducing kernels. However, unlike in one variable, this map is not automatically surjective. Theorem $3.2$ of \cite{D19} states that the isometric embedding $J_\alpha$ constructed above is unitary if and only if the polydisk algebra $A(\mathbb{D}^d)$ is dense in $L^2(\sigma_\alpha)$.

Now, let $\phi \colon \mathbb{D}^d\to \mathbb{D}$ be a rational inner function, with associated Clark measure $\sigma_{\alpha}$ for a fixed $\alpha \in \mathbb{T}$.
As is asserted in \cite{D19}, each $\sigma_{\alpha}$ is supported on the {\it unimodular level set}
\begin{equation}
\mathcal{C}_{\alpha}(\phi)=\mathrm{clos}\left\{\zeta \in \mathbb{T}^d: \lim_{r \to 1^-} \phi(r \zeta) = \alpha \right\},
\label{Clevelset}
\end{equation}
where ``clos'' denotes the closure of the set. When the function $\phi$ is clear from the context, we sometimes refer to this set as simply $\mathcal{C}_\alpha$. While the measure-support statement should be well known to specialists, we give a proof for the sake of completeness.
\begin{lemma} \label{DoubClaim}
Let $\phi$ be an RIF on $\mathbb{D}^d$ and let $\alpha \in \mathbb{T}$. Then $\mathrm{supp}(\sigma_{\alpha})\subset \mathcal{C}_{\alpha}(\phi)$. 
\end{lemma}

\begin{proof}
Let $B \subset \mathbb{T}^d$ be an open ball such that $\lim_{r \to
 1^-} \phi(r \zeta) \neq \alpha$ for all $\zeta \in B$. We need to show that $\sigma_\alpha(B) = 0$. Since the Poisson kernel is non-negative, we have that
\[
\int_B P(r\zeta, \eta) d \sigma_\alpha(\eta) \leq \int_{\mathbb{T}^d} P(r\zeta, \eta) d \sigma_\alpha(\eta) = \frac{1-|\phi(r\zeta)|^2}{|\alpha-\phi(r\zeta)|^2}
\]
for all $\zeta \in B$ and every $0 \leq r < 1$. By \cite[Corollary $14.6$]{Kne15}, the right hand side vanishes when $r$ tends to $1$, and so 
\[
\lim_{r \to 1^-} \int_B P(r\zeta, \eta) d \sigma_\alpha(\eta) = 0.
\]
Now consider the set 
\[
D_r(\zeta) := \{\eta \in \mathbb{T}^d: |r\zeta_j - \eta_j| \leq 2(1-r), \quad j=1, \ldots, d \}.
\]
For every $\eta$ in this set, we have that
\[
|r\zeta_j - \eta_j|^2 \leq 4 (1-r)^2 \to \frac{1-r^2}{4 (1-r)^2} = \frac{1+r}{4(1-r)} \leq \frac{1-r^2}{|r\zeta_j - \eta_j|^2},
\]
and so
\[
\left( \frac{1+r}{4(1-r)} \right)^d \leq P(r\zeta, \eta).
\]
Clearly
\[
\left( \frac{1+r}{4(1-r)} \right)^d \sigma_\alpha(B \cap D_r(\zeta)) \leq \int_{B \cap D_r(\zeta)} P(r\zeta, \eta) d \sigma_\alpha(\eta) \leq \int_B P(r\zeta, \eta) d \sigma_\alpha(\eta),
\]
so
\begin{equation} \label{density}
\lim_{r \to 1^{-}} \frac{\sigma_\alpha(B \cap D_r(\zeta))}{(1-r)^d} = 0.
\end{equation}
Note that since $\zeta_j, \eta_j \in \mathbb{T}$, the inequality
\[
|r \zeta_j - \eta_j| = |r- \eta_j\overline{\zeta}_j| < 2(1-r), 
\]
can be written in polar coordinates (with $e^{i\theta_j} = \eta_j\overline{\zeta}_j$) as
\[ \begin{aligned}
2(1-r) > |r- e^{i \theta_j}| &\iff 4(1-2r + r^2) > 1+ r^2 - 2r \cos \theta_j  \\
&\iff 3 r^2 - 6r + 3 = 3 (1-r)^2 > 2r - 2r \cos(\theta_j)\\
& \iff \cos(\theta_j) > 1 - \frac{3(1-r)^2}{2r},
\end{aligned}
\]
and so
\[
D_r(\zeta) = \left\{ \zeta e^{i\theta} \in \mathbb{T}^d: |\theta_j| < \cos^{-1}\left(1 - \frac{3(1-r)^2}{2r} \right),\, j=1, \ldots , d \right\}.
\]
In particular, as a subset of $\mathbb{T}^d$, this is a product of $d$ copies of the same interval, and so (for $r$ close to $1$) the Lebesgue measure of $D_r(\zeta)$ can be estimated independently of $\zeta$ by
\[
|D_r(\zeta)| = 2^d \cos^{-1}\left(1 - \frac{3(1-r)^2}{2r}\right)^d \geq c(d) \sqrt{\frac{3(1-r)^2}{2r}}^d \geq c'(d) (1-r)^d.
\]
Together with \eqref{density} this implies that
\[
\lim_{r \to 1^{-}} \frac{\sigma_\alpha(B \cap D_r(\zeta))}{|D_r(\zeta)|} = 0
\]
for every $\zeta \in B$. 

Since $D_r(\zeta)$ is a cube in $\mathbb{T}^d$ with volume tending to zero, this implies that the $d$-dimensional upper density of the restriction measure $(\sigma_\alpha)_{|B}$, defined by $(\sigma_\alpha)_{|B}(A) := \sigma_\alpha(B \cap A)$, is zero at every point in $\mathbb{T}^d$, see for example, the ideas around Proposition 2.2.2 in \cite{Krantz2}. This in turn implies that $(\sigma_\alpha)_{|B}$ is equal to zero which in particular implies that $\sigma_\alpha(B) = 0$. 
\end{proof}

Note that Lemma \ref{DoubClaim} implies that every Clark measure associated to an RIF is a singular measure with respect to the Lebesgue measure on $\mathbb{T}^d$. It is also worth noting that in the case where $\phi=\tilde{p}/p$ is a two-variable RIF, we actually have  $\mathrm{supp}(\sigma_{\alpha})= \mathcal{C}_{\alpha}(\phi)$. This will follow from our later results Theorem \ref{thm:2clarkformula} and Theorem \ref{thm:exceptional}. Thus, it makes sense to conjecture that $\mathrm{supp}(\sigma_{\alpha}) = \mathcal{C}_{\alpha}(\phi)$ for general RIFs on the polydisk $\mathbb{D}^d$ as well.

For the sake of completeness, we also state and prove the following converse, which is well known in the one-variable setting. 
\begin{lemma}
Let $\mu$ be a positive pluriharmonic measure on $\mathbb{T}^d$ with mass $1$. Then there is a holomorphic function $\phi_\mu\colon \mathbb{D}^d \to \mathbb{D}$ such that $\mu$ is the Aleksandrov-Clark measure corresponding to the holomorphic function $\phi_\mu$ and the parameter value $\alpha=1$. 

If $\mu$ is singular with respect to Lebesgue measure, then $\phi_\mu$ is an inner function and $\mu$ is its Clark measure for $\alpha=1$.
\end{lemma}
\begin{proof}
Let $H_\mu(z)$ be the holomorphic function on $\mathbb{D}^d$ whose real part is the Poisson integral of $\mu$ and which satisfies that $H_\mu(0)=1$. Such a function exists since the real part of the Poisson integral will be $1$ at the origin since $\mu$ is a probability measure, and we can choose a harmonic conjugate which vanishes at the origin.  

Now consider the function
\[
\phi_\mu(z) := \frac{H_\mu(z) - 1}{H_\mu(z)+1}.
\]
We have that
\[
H_\mu(z) = \frac{1+ \phi_\mu(z)}{1-\phi_\mu(z)},
\]
and so
\begin{align} \label{alek_poiss}
\frac{1 - |\phi_\mu(z)|^2}{|1-\phi_\mu(z)|^2} = \Re \left( \frac{1+ \phi_\mu(z)}{1-\phi_\mu(z)} \right)  = \Re(H_\mu(z)) = \int_{\mathbb{T}^d} P(z,\zeta) d \mu(\zeta).
\end{align}
Since $H_\mu$ maps $\mathbb{D}^d$ to the right half plane, and since $z \mapsto (z-1)/(z+1)$ maps the right half plane to the unit disc, we see that
$\phi_\mu\colon \mathbb{D}^d \to \mathbb{D}.$
Thus, \eqref{alek_poiss} shows that $\mu$ is the Aleksandrov-Clark measure corresponding to the holomorphic function $\phi_\mu(z)$ and $\alpha = 1$. 

If $\mu$ is singular with respect to Lebesgue measure, Theorem $2.3.1$ in \cite{Rud69} shows that
\[
\lim_{r \to 1^-} \int_{\mathbb{T}^d} P(rz,\zeta) d \mu(\zeta) = \lim_{r \to 1^-} \frac{1 - |\phi_\mu(rz)|^2}{|1-\phi_\mu(rz)|^2} = 0
\]
for almost every $z \in \mathbb{T}^d$, which shows that $|\phi_\mu(z)| = 1$ almost everywhere on $\mathbb{T}^d$.
\end{proof}
\subsection{Background on rational inner functions}
We shall need some detailed results concerning RIFs in two variables, but we begin by recalling some basic facts from the general theory.
We say that $p\in \mathbb{C}[z_1,\ldots, z_d]$ is a {\it stable polynomial} if $p$ has no zeros in $\mathbb{D}^d$. A polynomial in $d$ variables has {\it polydegree} 
$(n_1, \ldots, n_d)\in \mathbb{N}^d$ if $p$ has degree $n_j$ when viewed as a polynomial in the variable $z_j$. A result of Rudin and Stout \cite{RudSt65,Rud69} states that any RIF in $\mathbb{D}^d$ 
can be written in the form
\[\phi(z)=e^{ia}z_1^{k_1}\cdots z_d^{k_d}\frac{\tilde{p}(z)}{p(z)}\]
where $a\in \mathbb{R}$, $k_1,\ldots, k_d$ are natural numbers, $p$ is a stable polynomial of polydegree $d$, and $\tilde{p}$ is its {\it reflection}
\[\tilde{p}(z)=z_1^{n_1}\cdots z_d^{n_d}\overline{p\left(\frac{1}{\bar{z}_1}, \ldots, \frac{1}{\bar{z_d}}\right)}.\]
We shall often assume that the RIFs we consider are of the form $\phi=\tilde{p}/p$, where $p$ is a stable polynomial that is atoral. The concept of atoral polynomials is discussed at length in \cite{AMS06,BKPS}, but for the present work, we just note that atoral implies that $p$ and $\tilde{p}$ have no common factors, and that the zero set of $p$, denoted $\mathcal{Z}(p),$ satisfies $\dim(\mathcal{Z}(p)\cap \mathbb{T}^d)\leq d-2$.

Let us summarize some important definitions and properties of RIFs. First, we say that a RIF $\phi = q/p$ has polydegree $(n_1,\ldots, n_d)$ if $p$ and $q$ have no common factors and for each $j$, $n_j$ is the maximum of the degrees of $p$ and $q$ when they are viewed as polynomials in the variables $z_j$. When we consider $\phi =\tilde{p}/p$, then the polydegree of $\phi$ will always agree with both the polydegree of its denominator $p$ and the polydegree of its numerator $\tilde{p}$.

If $\phi$ is a polydegree $(n_1,\ldots, n_d)$ RIF, then for any index $j$ and any fixed collection of points $\{\zeta_1, \ldots, \zeta_{j-1}, \zeta_{j+1}, \ldots, \zeta_d\}\subset \mathbb{T}$, we can consider the one-variable function $z_j\mapsto \phi(\zeta_1, \ldots, z_j, \ldots, \zeta_d)$. If $z_j\mapsto p(\zeta_1, \ldots, z_j, \ldots, \zeta_d)$ is not identically zero, then it vanishes at at most $n_j$ points on $\mathbb{D}$. Because $\phi$ is bounded on $\mathbb{D}^d$, these have to be common zeros of the numerator and denominator of $z_j\mapsto \phi(\zeta_1, \ldots, z_j, \ldots, \zeta_d)$. Thus, they cancel out and we are left with a rational function defined on $\mathbb{D}$ with at most a finite number of singularities on $\mathbb{T}$. Because $\phi$ is a RIF, this one-variable function must attain unimodular boundary values at almost every point on $\mathbb{T}$. Hence, it  is a finite Blaschke product of degree at most $n_j$. As shown in the lemma below, generically the degree is exactly $n_j$, but for certain values of $\zeta$, the degree can be strictly smaller than $n_j$. 

Furthermore, if we restrict to a RIF $\phi$ on $\mathbb{D}^2$, then \cite[Lemma 10.1]{Kne15} states that $\phi$ does not have any singularities on $\mathbb{T}\times \mathbb{D}$ or $\mathbb{D} \times \mathbb{T}$. Thus, in that case for any $\zeta_1 \in \mathbb{T}$, the mapping $z_2\mapsto p(\zeta_1, z_2)$ can never vanish identically, so this slicing operation always yields a finite Blaschke product.

To prove the lemma below, we need some short-hand notation. Given a point $\zeta=(\zeta_1, \ldots, \zeta_{d-1}, \zeta_d) \in \mathbb{T}^d$, let us write $\zeta'=(\zeta_1, \ldots, \zeta_{d-1})\in \mathbb{T}^{d-1}$; we also use analogous notation for points $z\in \mathbb{C}^d$. 

\begin{lemma}\label{lem:degblaschke}
Let $\phi=\frac{\tilde{p}}{p}$ be an RIF on $\mathbb{D}^d$ with polydegree $(n_1,\ldots, n_d)$. For a fixed $\zeta'\in \mathbb{T}^{d-1}$, set $\phi_{\zeta'}(z_d)=\phi(\zeta_1,\ldots, \zeta_{d-1}, z_d)$. If $\phi$ does not have a singularity with coordinates of the form $(\zeta', \tau)\in \mathbb{T}^d$ for some $\tau \in \mathbb{T}$, then $\phi_{\zeta'}$ is a finite Blaschke product of degree $n_d$.
\end{lemma}
\begin{proof}
First, observe that if $\phi$ has no singularities of the form $(\zeta', \omega)\in \mathbb{T}^d$, then the function $p_{\zeta'}(z_d):=p(\zeta',z_d)$ is not identically zero. Then the assertion that $\phi_{\zeta'}$ is a finite Blaschke product of degree at most $n_d$ is immediate from the discussion proceeding the statement of Lemma \ref{lem:degblaschke}.

It remains to show that $\phi_{\zeta'}$ has degree exactly $n_d$. We first show that its initial numerator $\tilde{p}(\zeta', z_d)$ has degree $n_d$ and then argue that there can be no degree drop by canceling terms from the numerator and denominator. To this end, let us write \[p(z)=p_1(z')+z_dp_2(z',z_d)=p_1(z')+Q(z)\]
for polynomials $p_1$, $p_2$, and $Q$. Then 
\[\tilde{p}(z)=z_d^{n_d}\tilde{p}_1(z')+\tilde{Q}(z),\]
where the reflection of $p_1$ is only with respect to the variables $z_1, \dots, z_{d-1}$. From this, one can see that $\deg_{z_d}(\tilde{Q})<n_d$, using the definition of the ``$\sim$'' operation combined with the fact that each term in $Q$ has degree at least $1$ in the variable $z_d$. 

Next, we note that if $\tilde{p}_1(\zeta')=0$ for some $\zeta'\in\mathbb{T}^{d-1}$ then we would also have $p_1(\zeta')=0$. This in turn would imply that $p_{\zeta'}(0)=0$.  An application of Hurwitz's theorem as in \cite[p. 1123]{BPSmulti} implies that $p_{\zeta'}$ is either nonvanshing on $\mathbb{D}$ or identically zero. We have already established that $p_{\zeta'}$ is not identically zero and so, $p_{\zeta'}(0)=0$ would give a contradiction.  Thus, $\tilde{p}_1(\zeta') \ne 0$.

Hence, for $\zeta'\in \mathbb{T}^{d-1}$, we have $\deg \tilde{p}(\zeta',z_d)=n_d$. This means that any degree drop in $\phi_{\zeta'}$ must arise from cancelling a common zero of $\tilde{p}(\zeta',z_d)$ and $p(\zeta',z_d)$. Because  $p_{\zeta'}$ is nonvanishing on $\mathbb{D}$, this zero must necessarily occur on $\mathbb{T}$, which in turn would imply that $\phi$ has a singularity at some $(\zeta',\tau) \in \mathbb{T}^d$, contrary to our hypothesis. Thus, it must be the case that the degree of $\phi_{\zeta'}$ is exactly $n_d$.
\end{proof}

One useful way of studying RIFs is via their level sets $\mathcal{C}_{\alpha}(\phi)$ as defined in \eqref{Clevelset}. For example,   Lemma \ref{DoubClaim} shows their relevance to the analysis of the Clark measures associated with an RIF. In \cite{BPSmulti}, the authors established the following useful alternative description of the unimodular level sets.
\begin{theorem}
Let $\phi = \frac{\tilde{p}}{p}$ be an RIF on $\mathbb{D}^d$, fix $\alpha \in \mathbb{T}$, and set \[\mathcal{L}_{\alpha}(\phi)=\{\zeta \in \mathbb{T}^d\colon \tilde{p}(\zeta)-\alpha p(\zeta)=0\}.\] Then $\mathcal{C}_{\alpha}(\phi)=\mathcal{L}_{\alpha}(\phi)$.
\end{theorem}
\begin{proof}
See \cite[Theorem 2.6]{BPSmulti}.
\end{proof}

Much of the remainder of this paper will be concerned with Clark measures for rational inner functions on the bidisk $\mathbb{D}^2$. One reason why we focus on this case is that level sets of two-variable RIFs have much better properties than those of their $d$-variable counterparts. Namely, when $d\geq 3$, the level sets of $d$-dimensional RIFs can exhibit discontinuities. See \cite{BPSmulti} for a fuller discussion of the sometimes pathological nature of level sets for $d$-variable RIFs in dimension $d\geq3$. By contrast, when $d=2$ we have the lemma given below, which is implicit in \cite{BPS17, BPSprep}. As mentioned earlier, here and throughout the paper, we assume that a bidegree $(m,n)$ RIF has both $m>0$ and $n>0$.
\begin{lemma}\label{lem:levsets}
Let $\phi$ be a bidegree $(m,n)$ RIF. For each $\alpha \in \mathbb{T}$ and any choice of $\tau_0 \in \mathbb{T}$, there exist functions $g^\alpha_1, \ldots, g^\alpha_n$ defined on $\mathbb{T}$ and analytic on $\mathbb{T} \setminus \{\tau_0\}$ such that $\mathcal{C}_{\alpha}(\phi)$ can be written as a union of graphs of the form
\[\{(\zeta,g^\alpha_j(\zeta)) \colon \zeta \in \mathbb{T}  \}, \quad j=1,\ldots, n,\]
potentially, together with a finite number of vertical lines $\zeta_1=\tau_1, \ldots, \zeta_1=\tau_k$, where each $\tau_j \in \mathbb{T}$. 
\end{lemma}
\begin{proof} We first fix $\tau \in \mathbb{T}$ and obtain a parameterization of $\mathcal{C}_\alpha \cap (I_\tau \times \mathbb{T})$, where $I_\tau$ is a small interval in $\mathbb{T}$ containing $\tau$. We have to consider both the situation where $\tau$ is \emph{not} the  $z_1$-coordinate of a singularity of $\phi$ (Step 1) and the situation where $\tau$ \emph{is}  the  $z_1$-coordinate of a singularity of $\phi$  (Step 2). In the latter case, we reference previous results to obtain the parameterization. Finally, we glue these local parameterizations together to obtain global ones (Step 3). 

{\it Step 1.}
First, let us assume that $\tau$ is {\it not} the $z_1$-coordinate of a singularity of $\phi $ on $\mathbb{T}^2$. Then, Lemma \ref{lem:degblaschke} implies that $\phi_{\tau}(z_2):=\phi(\tau, z_2)$ is a nonconstant finite Blaschke product with $\deg \phi_{\tau}=n$. By properties of nonconstant finite Blaschke products, there are precisely $n$ distinct points $\eta_1, \ldots, \eta_n\in \mathbb{T}$ such that $\phi_{\tau}(\eta_j)=\alpha$ for $j=1,\ldots, n$. Since $\phi_{\tau}$ is a non-constant Blaschke product, $\phi_{\tau}'(\zeta)\neq 0$ for all $\zeta \in \mathbb{T}$, and then 
\[\frac{\partial \phi}{\partial z_2}(\tau, \eta_j)=\phi'_{\tau}(\eta_j)\neq 0, \quad j=1, \ldots, n.\]
Since the two-variable function $\phi$ is analytic in a neighborhood of each $(\tau, \eta_j)$, the implicit function theorem applies and yields locally analytic functions $g^\alpha_{1,{\tau}}, \ldots, 
g_{n,{\tau}}^\alpha$ and an open interval $I_\tau$ containing $\tau$ such that $\mathcal{C}_{\alpha}$ is parametrized by
\begin{equation}
\zeta_2=g_{1,{\tau}}^\alpha(\zeta_1), \quad \ldots, \quad \zeta_2=g_{n, \tau}^\alpha(\zeta_1)
\label{locpara}
\end{equation}
on $I_\tau \times U$, where $U$ is initially a union of open arcs containing the points $\eta_{1},\ldots, \eta_{n}$. By shrinking the interval $I_\tau$ further, we can ensure that \eqref{locpara} parametrizes all pieces of $\mathcal{C}_{\alpha}$ that are contained in the strip $I_\tau \times \mathbb{T}$, since for each $\zeta_1$ close to $\tau$, we can ensure that the equation $\phi(\zeta_1, \zeta_2)=\alpha$ has exactly $n$ distinct solutions.

{\it Step 2.} Suppose now that $\tau$ is the $z_1$-coordinate of a singularity of $\phi$. Then either (a) the line $\{\zeta\in \mathbb{T}^2\colon \zeta_1=\tau\}$ is contained in $\mathcal{C}_{\alpha}$, or (b) the intersection of the line $\{\zeta \in \mathbb{T}^2\colon \zeta_1=\tau\}$ with $\mathcal{C}_{\alpha}$ consists of at most $n$ points coming from the singularities of $\phi$ that have $z_1$-coordinate $\tau$ as well as additional points $\eta \in \mathbb{T}$ with $\phi_{\tau}(\eta)=\alpha$.

Let us address case (a) first. Basically, we need to parameterize any pieces of $\mathcal{C}_{\alpha}$ that intersect the line $\{\zeta\in \mathbb{T}^2\colon \zeta_1=\tau\}$. To that end, assume
that $(\tau, \gamma)\in \mathbb{T}^2$ is the limit of a sequence of points $(\tau_m, \gamma_m) \subset\mathcal{C}_{\alpha}$ with $\tau_m\neq \tau$. We claim that $(\tau, \gamma)$ must be a singularity of $\phi$, which will allow us to apply known results. To that end, for each $m$, define the one-variable function $\phi_{m}(z_1)=\phi(z_1, \gamma_m\bar{\tau}_mz_1)$.
We have $\phi_{m}(\tau)=\alpha$ since the vertical line $\{\zeta_1=\tau\}$ was assumed to belong to $\mathcal{C}_{\alpha}$, and moreover $\phi_m(\tau_m)=\phi(\tau_m, \gamma_m)=\alpha$ by assumption. Since $\phi_m$ is a nonconstant finite Blaschke product, for any given $\lambda\in \mathbb{T} \setminus \{\alpha\}$, we can find a sequence $(\rho_m) \subseteq \mathbb{T}$ with each $\rho_m$ on the smaller of the two arcs of $\mathbb{T}$ between $\tau$ and $\tau_m$ with  the property that $\phi_{m}(\rho_m)=\lambda$. Since $\tau_m \rightarrow \tau$, we must also have $\rho_m \rightarrow \tau$.  Then $\phi(\rho_m, \gamma_m\bar{\tau}_m \rho_m)=\lambda$ for each $m$. Since $(\rho_m, \gamma_m\bar{\tau}_m \rho_m)\to (\tau, \gamma)$ as $m\to \infty$, this implies that $\phi$ is discontinuous at $(\tau, \gamma)$. Hence, $\phi$ has a singularity at $(\tau, \gamma)$.
This means that we can apply \cite[Theorem 2.9]{BPSprep} at $(\tau, \gamma)$, which states that $\mathcal{C}_{\alpha}$ can be locally parameterized by analytic functions near each singularity of $\phi$.

If we are in case (b), we can again parameterize $\mathcal{C}_{\alpha}$ at the singularities using \cite[Theorem 2.9]{BPSprep}, and apply the implicit function theorem at the other points since $\phi_{\tau}$ is again non-constant.

Thus in both case (a) and case (b) we get a collection of analytic functions which, possibly together with a vertical line $\{\zeta_1=\tau\}$, parameterize $\mathcal{C}_{\alpha}$ on some strip $I_\tau \times \mathbb{T}$, provided $I_\tau$ is chosen to be a sufficiently small interval containing $\tau$. Furthermore, for all but finitely many $\tau$,  there are precisely $n$ distinct points $\eta_1, \ldots, \eta_n\in \mathbb{T}$ such that $\phi(\tau, \eta_j)=\alpha$. This means that in each case, we must get exactly $n$ functions.

{\it Step 3.}
We can now cover $\mathbb{T}^2$ with a union of strips of the form $I_{\tau}\times \mathbb{T}$, where each $I_{\tau}$ is from Step $1$ or Step $2$. Since there are finitely many singularities, and since $\mathbb{T}^2$ is compact, we can refine this to a finite number of strips in such a way that each singularity of $\phi$ is inside one of these strips. On each strip we have an analytic parameterization, and on their overlaps the parameterizations must agree. The one difficulty is that as we go all the way around $\mathbb{T}$, one branch might end at the point where another branch began and so, it might not be the case that $g^\alpha_j(e^{i\theta}) = g^\alpha_j(e^{i\theta +2\pi i})$ for each $j$. Instead we might get  $g^\alpha_j(e^{i\theta}) = g^\alpha_k(e^{i\theta +2\pi i})$ with $j \ne k$. Thus, we need to allow one $\tau_0 \in \mathbb{T}$ where the branches can jump. With that technicality, we can 
 obtain functions $g^\alpha_1, \ldots, g^\alpha_n$ that are globally defined on $\mathbb{T}$, parameterize the components of $\mathcal{C}_\alpha(\phi)$ that are not lines, and are analytic except at a single point.
\end{proof} 
\begin{remark}\label{rem:nosings}
If $\phi$ is a two-variable RIF which has no singularities, then Step 2 becomes superfluous, and the conclusion follows from Steps 1 and 3. But these steps, unlike Step 2, do not require us to restrict to dimension $d=2$. 

Hence, if $\phi=\frac{\tilde{p}}{p}$ is a $d$-variable RIF, $d\geq 2$, with $\deg_{z_d}p=n_d$ and with no singularities on $\overline{\mathbb{D}}^d$, then there exist analytic functions $g^\alpha_1,\ldots g^\alpha_{n_d}$ such that
$\mathcal{C}_{\alpha}$ can be parameterized as
\[\zeta_{d}=g^\alpha_1(\zeta_1,\ldots, \zeta_{d-1}), \ldots, \zeta_d=g^\alpha_{n_d}(\zeta_1,\ldots, \zeta_{d-1}).\]
We will use this parameterization in our later investigations of the $d$-variable situation.
\end{remark}

Lastly, our fine analysis of Clark measures for two-variable RIF will require the notion of {\it contact order} of a RIF at a singularity, a concept introduced in \cite{BPS17}, and further developed in \cite{BPSprep}, in connection with the study of integrability of the partial derivatives of a RIF. Let $\alpha_1, \alpha_2\in \mathbb{T}$ with $\alpha_1\neq \alpha_2$ and let $\{g_j^{\alpha_1}\}_j$ and $\{g^{\alpha_2}_k\}_k$ be the functions from Lemma \ref{lem:levsets} associated with $\alpha_1$ and $\alpha_2$ respectively.  Then \cite[Theorem 3.1]{BPSprep} implies the following.
\begin{lemma}\label{lem:CO}
Excluding at most one $\alpha_0\in \mathbb{T}$, the {\it contact order} of a RIF at a singularity $(\tau, \gamma)\in \mathbb{T}^2$ is the maximal order of vanishing of the pairwise differences $g^{\alpha_1}_j(\zeta)-g^{\alpha_2}_k(\zeta)$ at $\zeta=\tau$ for any pair $\alpha_1, \alpha_2\in \mathbb{T}\setminus\{\alpha_0\}$, where we restrict attention to the $g^{\alpha_i}_j$ that satisfy $g^{\alpha_i}_j(\tau)= \gamma$.
\end{lemma}
We note that it follows from the work in \cite{BPSprep} that the contact order of a RIF at a singularity is always a positive even integer. Also, while the computation in Lemma \ref{lem:CO} might make it look like contact order somehow depends on the choice of the constants $\alpha_1, \alpha_2 \in \mathbb{T}$, it is actually independent of that choice.

\section{Structure of Clark measures for RIFs}\label{sec:structone}

In this section, we determine the structure of the Clark measures $\sigma_\alpha$ for general RIFs on $\mathbb{D}^2$. There are two cases to consider: the case where the parameter $\alpha$ is generic and the case where $\alpha$ is exceptional. These two types of parameters are defined as follows.
\begin{definition} 
A point $\alpha \in \mathbb{T}$ is said to be an {\it exceptional value} if $\phi(\tau, z_2)\equiv \alpha$ or if $\phi(z_1,\tau)\equiv\alpha$ for some $\tau \in \mathbb{T}$. This is equivalent to saying that one of the two lines  $\{\zeta\in \mathbb{T}^2\colon \zeta_1=\tau\}$  or  $\{\zeta\in \mathbb{T}^2\colon \zeta_2=\tau\}$ is in $\mathcal{C}_\alpha(\phi)$ for some $\tau\in \mathbb{T}$. If $\alpha\in \mathbb{T}$ is not an exceptional value, then we say that $\alpha$ is a {\it generic value}.
\end{definition}
\begin{remark}
For bidegree $(n,1)$ RIFs, it was shown in \cite[Section 3]{BCSMich} that $\alpha \in \mathbb{T}$ is exceptional if and only if $\alpha$ is the non-tangential value of $\phi$ at some singularity of $\phi$. However, this characterization does not generalize to higher-degree RIFs. 

Still, there are RIFs with bidegree at least $(2,2)$ with exceptional values. In particular, consider
\[\phi(z)=\frac{2z_1^2z_2^2-z_1^2-z_2^2}{2-z_1^2-z_2^2}.\]
If we set $\alpha =-1$, then $\mathcal{C}_\alpha(\phi)$ contains the four lines $\{\zeta\in \mathbb{T}^2\colon \zeta_1=\pm1\}$ and $\{\zeta\in \mathbb{T}^2\colon \zeta_2= \pm1\}$ and so $\alpha =-1$ is an exceptional value for $\phi$.
\end{remark}

After looking at the two-variable generic case, we will also show how one can translate some of those arguments to the $d$-variable setting.

\subsection{Clark measures in the generic two-variable case.}
Our first goal is to prove the follow description of the Clark measures $\sigma_{\alpha}$ for \emph{generic} parameter values $\alpha \in \mathbb{T}$.

\begin{theorem}\label{thm:2clarkformula}
Let $\phi=\frac{\tilde{p}}{p}$ be a bidegree $(m,n)$ RIF, and let $\alpha \in \mathbb{T}$ be generic for $\phi$. Then, for $f\in C(\mathbb{T}^2)$, the associated Clark measure $\sigma_{\alpha}$ satisfies
\[\int_{\mathbb{T}^2}f(\zeta)d\sigma_{\alpha}(\zeta)=\sum_{j=1}^n\int_{\mathbb{T}}f(\zeta,g^\alpha_j(\zeta))\frac{dm(\zeta)}{|\frac{\partial \phi}{\partial z_2}(\zeta, g^\alpha_j(\zeta))|},\]
where $g^\alpha_1,\ldots, g^\alpha_n$ are the parametrizing functions for $\mathcal{C}_{\alpha}(\phi)$ from Lemma \ref{lem:levsets}.
\end{theorem}
\begin{proof}
Let $\alpha  \in \mathbb{T}$ be generic for $\phi$. Then by Lemma \ref{lem:levsets}, there exist functions $g^\alpha_1, \ldots, g^\alpha_n$ analytic on $\mathbb{T}$ (minus some arbitrary base point $\tau_0$) such that 
\[\mathcal{C}_{\alpha}(\phi)=\bigcup_{j=1}^n\{(\zeta, g^\alpha_j(\zeta))\colon \zeta \in \mathbb{T}\}.\]
We first establish the desired formula in the special case where $f$ is a product of one-variable Poisson kernels. Fix $z_2\in \mathbb{D}$ and consider the one-variable function 
\begin{equation}  \label{eqn:slice0} \varphi_{z_2}(z_1)=\frac{1-|\phi(z_1,z_2)|^2}{|\alpha-\phi(z_1,z_2)|^2} = \int_{\mathbb{T}^2} P_{z_1}(\zeta_1)P_{z_2}(\zeta_2) d\sigma_{\alpha}(\zeta), \quad z_1\in \mathbb{D}.\end{equation} 
Because $\phi$ is a two-variable RIF, it has no singularities on $\mathbb{T} \times \mathbb{D}$ and so, $\phi(\cdot, z_2)$ is continuous on $\overline{\mathbb{D}}$. Moreover, by the discussion preceding Lemma \ref{lem:degblaschke}, for each $\zeta \in \mathbb{T}$, the function $\Phi_{\zeta}:=\phi(\zeta,\cdot)$ is a finite Blaschke product. If for some $\zeta$ we had 
\[ \Phi_{\zeta}(z_2) = \phi(\zeta, z_2) = \alpha,\]
then $\Phi_\zeta$ would be constant on $\overline{\mathbb{D}}$ and that would imply that $\alpha$ is an exceptional value, a contradiction. Thus, $\phi(\cdot, z_2)$  cannot attain  the value $\alpha$ in $\overline{\mathbb{D}}$.  Then, the function $\varphi_{z_2}$ is continuous on $\overline{\mathbb{D}}$ and thus, $\varphi_{z_2}$ is the Poisson integral of its boundary values. In other words, for each $z_1 \in \mathbb{D}$, 
\begin{equation}
\varphi_{z_2}(z_1)=\int_{\mathbb{T}}\frac{1-|\phi(\zeta,z_2)|^2}{|\alpha-\phi(\zeta, z_2)|^2}P_{z_1}(\zeta)dm(\zeta).
\label{sliceone}
\end{equation}
 For all but finitely many $\zeta$, Lemma \ref{lem:degblaschke}  implies that the finite Blaschke product $\Phi_{\zeta}$ has degree $n$.  By standard one-variable results, see \cite{CMR,GMR}, the Clark measure for $\Phi_\zeta$ is given by 
\[ \sum_{j=1}^n\frac{1}{|\Phi'_{\zeta}(\eta_j)|}\delta_{\eta_j},\]
 where $\{\eta_1,\ldots, \eta_n\}\subset \mathbb{T}$ are the distinct points on $\mathbb{T}$ with $\Phi_{\zeta}(\eta_j)=\alpha$. We note that $\Phi'_{\zeta}(z_2)=\frac{\partial \phi}{\partial z_2}(\zeta,z_2)$. Then the parametrization of $\mathcal{C}_{\alpha}(\phi)$ given above implies that
\begin{equation}
\frac{1-|\phi(\zeta,z_2)|^2}{|\alpha-\phi(\zeta,z_2)|^2}=\sum_{j=1}^n\frac{1}{|\frac{\partial \phi}{\partial z_2}(\zeta, g^\alpha_j(\zeta))|}P_{z_2}(g^\alpha_j(\zeta)).
\label{slicetwo}
\end{equation}
Combining \eqref{eqn:slice0}, \eqref{sliceone}, and \eqref{slicetwo}, we obtain the desired formula for $f=P_{z_1}P_{z_2}$. 

The conclusion of the theorem now follows from the fact that linear combinations of Poisson kernels are dense in $C(\mathbb{T}^2)$.
\end{proof}

\begin{remark}
One can interchange the roles of the variables $z_1$ and $z_2$ to obtain an analogous version of Theorem \ref{thm:2clarkformula} where $\mathcal{C}_{\alpha}$ is parametrized using the variable 
$\zeta_2$. See \cite{BKPS} for an in-depth discussion concerning variable switching.
\end{remark}

\subsection{Clark measures for RIFs in more than two variables.}

Let us take a brief interlude to examine how the arguments in the previous section generalize to RIFs in more than two variables. It turns out that singularities present significant complications (discussed more below) so instead, let us first assume that we have a $d$-variable RIF with no singularities on the closed polydisk. Then we have the following result.

\begin{theorem}\label{thm:dclarkformula}
Let $\phi=\frac{\tilde{p}}{p}$ be a polydegree $(n_1,\dots, n_d)$ RIF with no singularities on $\overline{\mathbb{D}}^d$ and let $\alpha \in \mathbb{T}$. 
Then, for $f\in C(\mathbb{T}^d)$, the associated Clark measure $\sigma_{\alpha}$ satisfies
\[\int_{\mathbb{T^d}}f(\zeta)d\sigma_{\alpha}(\zeta)=\sum_{j=1}^{n_d}\int_{\mathbb{T}^{d-1}}f(\zeta', g^\alpha_j(\zeta'))\frac{dm(\zeta')}{|\frac{\partial \phi}{\partial z_d}(\zeta', g^\alpha_j(\zeta'))|},\]
where $g^\alpha_1,\ldots, g^\alpha_{n_d}$ are the analytic functions that parametrize $\mathcal{C}_{\alpha}(\phi)$ from Remark \ref{rem:nosings}.
\end{theorem} 

\begin{proof} The proof is basically the same as that of Theorem \ref{thm:2clarkformula}. Fix $z_d \in \mathbb{D}$ and define
\[\varphi_{z_d}(z')=\frac{1-|\phi(z',z_d)|^2}{|\alpha-\phi(z',z_d)|^2}, \quad z'\in \mathbb{D}^{d-1}.\]
Then for $\zeta' \in \mathbb{T}^{d-1}$, Lemma \ref{lem:degblaschke} implies that $\Phi_{\zeta'}:=\phi(\zeta', \cdot)$ is a nonconstant finite Blachke product of degree $n_d$ and using that, we can conclude that $\phi(\cdot ,z_d)$ does not attain the value $\alpha$ in $\overline{\mathbb{D}}^{d-1}$. Then $\varphi_{z_d}$ is continuous on the closed polydisk and so we can write it as the Poisson integral of its boundary values
\begin{equation} \label{eqn:dvarclark} \frac{1-|\phi(z',z_d)|^2}{|\alpha-\phi(z',z_d)|^2} = \varphi_{z_d}(z')
=\int_{\mathbb{T}^{d-1}}\frac{1-|\phi(\zeta',z_d)|^2}{|\alpha-\phi(\zeta', z_d)|^2}P_{z'}(\zeta')dm(\zeta')
\end{equation}
for $z' \in \mathbb{D}^{d-1}$. Furthermore,  $\Phi_{\zeta'}$ has associated Clark measure
\[\sum_{j=1}^{n_d}\frac{1}{|\Phi'_{\zeta'}(\eta_j)|} \delta_{\eta_j} =\sum_{j=1}^{n_d}\frac{1}{|\frac{\partial \phi}{\partial z_d}(\zeta',\eta_j)|} \delta_{\eta_j} ,\]
where $\{\eta_1,\ldots, \eta_{n_d}\}\subset \mathbb{T}$ are the distinct points satisfying $\Phi_{\zeta'}(\eta_j)=\alpha$. As in the proof of Theorem \ref{thm:2clarkformula}, we can then rewrite \eqref{eqn:dvarclark} using the one-variable Clark measure and the parameterizing functions from Remark \ref{rem:nosings} to obtain the desired equality when $f$ is a product of one-variable Poisson kernels. The conclusion of the theorem follows from the fact that linear combinations of Poisson kernels are dense in $C(\mathbb{T}^d)$.
\end{proof}

However, if $\phi$ is a polydegree $(n_1,\dots, n_d)$ RIF with singularities on the boundary of $\mathbb{D}^d$, then these arguments break down in multiple places. For example, even if $\alpha$ is generic in the natural sense, the existence of singularities means that we still cannot necessarily guarantee that $\varphi_{z_d}(\cdot)$ will be continuous on $\overline{\mathbb{D}}^{d-1}$. This means we cannot always perform the trick of rewriting that key function in terms of the Poisson integral of its boundary values. 

Similarly, in two variables, we were able to invoke Lemma \ref{lem:levsets} to deduce that the $\eta_j$ points  could be described by analytic functions, regardless of whether $\phi$ possessed singularities or not. In three or more variables, the analogous statement is false in general, see \cite{BPSmulti}; in that paper, the authors show that when $d=3$ and $\phi$ has singularities, the functions parameterizing the components of $\mathcal{C}_\alpha(\phi)$ need not be continuous. So, additional work appears to be needed to obtain a version of Theorem \ref{thm:2clarkformula} in the general $d$-variable setting.

\subsection{Clark measures in the exceptional two-variable case.}
We now return to two-variable RIFs  and examine Clark measures associated with the exceptional parameter values. In this setting, additional care is required to handle cancellations present in $\phi$ on the vertical line part of $\mathcal{C}_{\alpha}$ and control the $z_1$-partial derivative of $\phi$ on that vertical line. That is the context of the following lemma.
\begin{lemma}\label{lem:excvalueconst}
Let $\phi=\frac{\tilde{p}}{p}$ be a bidegree $(m,n)$ RIF and suppose that the line $\{\zeta\in \mathbb{T}^2\colon \zeta_1=\tau \}$ is in $\mathcal{C}_\alpha(\phi)$. Then $\frac{\partial \phi}{\partial z_1}(\tau, z_2)\equiv c_1$ for some constant $c_1\neq 0$.
\end{lemma}
\begin{proof} In this proof, we will carefully analyze how the presence of $\{ \zeta \in \mathbb{T}^2: \zeta_1 = \tau\}$ in $\mathcal{C}_\alpha$ affects the structure of the three polynomials $p$, $\tilde{p}$ and $\tilde{p}-\alpha p$. These polynomials show up when we compute $\frac{\partial \phi}{\partial z_1}(\tau, z_2)$, and we will use our findings to deduce that this partial derivative must be a nonzero constant.

First, the assumption that  $\{\zeta\in \mathbb{T}^2\colon \zeta_1=\tau \}$ is in $\mathcal{C}_\alpha(\phi)$ implies that 
\begin{equation} \label{eqn:MMP} \alpha p(\tau, \zeta_2)=\tilde{p}(\tau, \zeta_2)\end{equation}
for all $\zeta_2 \in \mathbb{T}$. Recall that $p$ has no zeros in $\mathbb{D}^2\cup (\mathbb{T}\times \mathbb{D})\cup (\mathbb{D}\times \mathbb{T})$, see \cite[Lemma 10.1]{Kne15}. Then \eqref{eqn:MMP} coupled with the maximum modulus principle implies that $\phi_\tau(z_2) :=\phi(\tau,z_2) \equiv \alpha$ on $\mathbb{D}$ as well. Thus,
\[ \alpha p(\tau, z_2)=\tilde{p}(\tau, z_2),\]
for all $z_2$ and arguments very similar to those in Lemma \ref{lem:degblaschke} imply that the degrees of those polynomials in $z_2$ must be $n$. Then by the above facts about the locations of the zeros of $p$, there must exist $\lambda_1,\ldots, \lambda_J\in \mathbb{T}$, integers $m_1, \ldots, m_J$ with $m_1+\cdots +m_J=n$, and polynomials $q_1, q_2$ of bidegree at most $(m-1,n)$ such that
\begin{equation}
p(z)=(z_1-\tau)q_1(z)+\prod_{j=1}^J(z_2-\lambda_j)^{m_j}
\label{taupdecomp}
\end{equation}
and
\[\tilde{p}(z)=(z_1-\tau)q_2(z)+\alpha\prod_{j=1}^J(z_2-\lambda_j)^{m_j}.\]
Before analyzing $\frac{\partial \phi}{\partial z_1}(\tau, z_2)$, we need to show that 
the order of vanishing of $p$ at $(\tau, \lambda_j)$ is equal to $m_j$. To see this, write $p(x_1+\tau,x_2+\lambda_j)$ using its homogeneous expansion
\[p(x_1+\tau, x_2+\lambda_j)=P_M(x_1,x_2)+\sum_{k\geq M+1}P_k(x_1,x_2), \]
where each $P_k$ is homogeneous of degree $k$, and $M$ is the order of vanishing of $p$ at $(\tau, \lambda_j)$. Now, as is explained in \cite[Section 2]{BKPS}, we must have 
\[P_M(x_1,x_2)=c\prod_{j=1}^M(x_2-a_jx_1)\]
for some $c\neq 0$ and $a_1, \ldots, a_M>0$. Then using  \eqref{taupdecomp}, we have
\[p(\tau+x_1, \lambda_j+x_2)=c\prod_{j=1}^M(x_2-a_jx_1)+\sum_{k\geq M+1}P_k(x_1,x_2)=x_2^{m_j}r(x_2)+x_1q_1(\tau+x_1, \lambda_j+x_2),\]
for some polynomial $r$ with $r(0)\neq 0$. Then, plugging in $x_1=0$, we get  
\[cx_2^{M}+\sum_{k\geq M+1}P_k(0,x_2)=x_2^{m_j}r(0).\]
For each $k \ge M+1$, either $P_{k}(0,x_2)=0$  or it vanishes to order strictly higher than $M$. Thus, the above equation gives $M=m_j$ as claimed.
Expanding $\tilde{p}$ in a similar fashion gives
\[\tilde{p}(\tau+x_1, \lambda_j+x_2)=\sum_{k\geq N}Q_k(x_1,x_2),\]
where $N$ is the order of vanishing of $\tilde{p}$ at $(\tau, \lambda_j)$.
 Using Proposition 14.5 and related results in \cite{Kne15}, we can conclude that $N=M=m_j$ and $Q_{m_j}=\alpha P_{m_j}$. 

This means that $\tilde{p}-\alpha p$ vanishes to order at least $m_j+1$ at $(\tau, \lambda_j)$. Using the previous equations for $p$ and $\tilde{p}$, we have
\[\tilde{p}(z)-\alpha p(z)=(z_1-\tau)R(z),\]
where $R$ vanishes to order at least $m_j$ at each $(\tau, \lambda_j)$ and $\deg R\leq (m-1,n)$. This latter condition means $\deg_{z_1} R \le m-1$ and  $\deg_{z_2} R \le n$, where these are the degrees of $R$ in $z_1$ and $z_2$ separately. Thus, the one-variable polynomial $R(\tau, z_2)$ vanishes to order at least $m_j$ at each $\lambda_j$.  Since $\deg R(\tau, z_2) \le n$, this means either $R(\tau, z_2)$ is identically zero or $\prod_{j=1}^J(z_2-\lambda_j)^{m_j}$ divides $R(\tau, z_2)$. The second case would actually imply that 
\begin{equation} \label{eqn:R2} R(\tau, z_2) = c_1\prod_{j=1}^J(z_2-\lambda_j)^{m_j},\end{equation}
for some $c_1 \ne 0$.

 Now we have enough information to study $\frac{\partial \phi}{\partial z_1}(\tau, z_2)$. Specifically, by canceling terms, we have
\[\frac{\partial \phi}{\partial z_1}(\tau, z_2)=\frac{\frac{\partial \tilde{p}}{\partial z_1}p-\tilde{p}\frac{\partial p}{\partial z_1}}{p^2}(\tau, z_2)=\frac{\frac{\partial }{\partial z_1}(\tilde{p}-\alpha p)}{p}(\tau, z_2).\]
Hence, implementing our previous observations gives
\[ \frac{\partial \phi}{\partial z_1}(\tau, z_2)=\frac{R(\tau, z_2)}{\prod_{j=1}^{J}(z_2-\lambda_j)^{m_j}}.\]
Since, for almost every $\zeta \in \mathbb{T}$, $\phi(\cdot, \zeta)$ is a finite Blaschke product, its derivative cannot vanish on $\mathbb{T}$. Hence $R(\tau, z_2)$ is not identically zero. Thus, it must be the case that $R(\tau, z_2)$ satisfies \eqref{eqn:R2} and so, $\frac{\partial \phi}{\partial z_1}(\tau, z_2)\equiv c_1 \ne 0$, as claimed.
\end{proof}

%
%

We will use a limiting argument to study Clark measures associated to exceptional values using known results for generic parameter values. A key ingredient is the following lemma.
\begin{lemma}\label{lem:unif} Let $\phi=\frac{\tilde{p}}{p}$ be a bidegree $(m,n)$ RIF and let $\tau_1, \dots, \tau_K$ denote the $z_1$-coordinates of the singularities of $\phi$ on $\mathbb{T}^2$. Let 
$\hat{\epsilon}<\frac{1}{2}\min\{|\tau_k-\tau_j|\colon j\neq k\}$ and define 
\begin{equation} \label{eqn:Se} S_{\hat{\epsilon}} := \left \{ \zeta \in \mathbb{T}: \min_{1 \le k \le K} |\zeta -\tau_k | < \hat{\epsilon}\right\}.\end{equation}
Then if $f \in C(\mathbb{T}^2)$, we have
\begin{equation} \label{eqn:Fazeta} F(\zeta, \alpha) : = \sum_{j=1}^n  f(\zeta, g_j^\alpha(\zeta)) \frac{1}{|\frac{\partial \phi}{\partial z_2}(\zeta, g^\alpha_j(\zeta))|}\end{equation}
is uniformly continuous on $(\mathbb{T} \setminus S_{\hat{\epsilon}}) \times \mathbb{T}$, where  $g^\alpha_1,\ldots, g^\alpha_n$ are the parametrizing functions for $\mathcal{C}_\alpha(\phi)$ from Lemma \ref{lem:levsets}.
\end{lemma}

\begin{proof} First, note that by standard properties of RIFs, $\frac{\partial \phi}{\partial z_2}$ is continuous and nonzero on  $(\mathbb{T} \setminus S_{\hat{\epsilon}}) \times \mathbb{T}$. So it suffices to show that the set $\{ g_1^\alpha(\zeta), \dots, g_n^\alpha(\zeta)\}$ is continuous on  $(\mathbb{T} \setminus S_{\hat{\epsilon}}) \times \mathbb{T}$ up to a reordering of the functions. Specifically, we need to show that for each $\epsilon_0>0,$ there is a $\delta >0$ such that if $|\zeta-z| + | \alpha - \gamma| < \delta$ for $\zeta, z \in (\mathbb{T} \setminus S_{\hat{\epsilon}})$ and $\alpha, \gamma \in \mathbb{T}$, then after potentially reordering, we have 
\begin{equation} \label{eqn:epsilon}   \sum_{j=1}^n \left| g_j^{\alpha}(\zeta) - g_j^{\gamma}(z) \right| < \epsilon_0.\end{equation}
To that end, we will use the implicit function theorem. Define the function $\Phi(\zeta, w, \alpha) = \phi(\zeta,w)-\alpha$. By properties of finite Blaschke products, for each fixed $(\zeta_0, \alpha_0) \in (\mathbb{T} \setminus S_{\hat{\epsilon}})\times \mathbb{T}$, there exist distinct $w_1, \dots, w_n \in \mathbb{T}$ such that $\Phi(\zeta_0, w_j, \alpha_0)=0$. Then for $j=1, \dots, n$, the implicit function theorem implies that there are open arcs in $\mathbb{T}$, which we denote $U_1^j:=U_1^j(\zeta_0, \alpha_0), U_2^j:=U_2^j(\zeta_0, \alpha_0), U_3^j:=U_3^j(\zeta_0, \alpha_0)$ centered at $\zeta_0, w_j, \alpha_0$ respectively and a continuous function $G_j^{(\alpha_0, \zeta_0)}$ such that 
\[ \left \{ (\zeta, w, \alpha) \in U_1^j \times U_2^j \times U_3^j: \Phi(\zeta,w,\alpha)=0 \right\} =\left \{  (\zeta, G_j^{(\alpha_0, \zeta_0)}(\zeta, \alpha), \alpha): (\zeta, \alpha) \in U^j_1\times U^j_3\right\}. \]
By shrinking these arcs if necessary, we can assume that the $U_1^j, U_3^j$ do not depend on $j$, that the $U_2^1, \dots, U_2^n$ are pairwise-disjoint, and that the $G$ functions are uniformly continuous on $U_1 \times U_3$.  Then the family of sets
\[ \left \{ U_1(\zeta_0, \alpha_0) \times U_3(\zeta_0, \alpha_0) : (\zeta_0, \alpha_0)\in (\mathbb{T} \setminus S_{\hat{\epsilon}}) \times \mathbb{T} \right\} \]
forms an open cover of $(\mathbb{T} \setminus S_{\hat{\epsilon}}) \times \mathbb{T}.$ Since  $(\mathbb{T} \setminus S_{\hat{\epsilon}}) \times \mathbb{T}$ is compact, we can obtain a finite  subcover 
\[  \big \{ U_1(\zeta_\ell, \alpha_\ell) \times U_3(\zeta_\ell, \alpha_\ell)\big\}_{\ell=1}^L.\]
Now choose $\delta >0$ such that if $|\zeta-z| + | \alpha - \gamma| < \delta$, then the points $(\zeta, \alpha), (z,\gamma)$ must be in at least one common set in this finite subcover. Shrinking $\delta$ if necessarily, we can further assume that for all of the $G_j^{(\alpha_\ell, \zeta_\ell)}$, if $(\zeta, \alpha), (z,\gamma)$ are in $U_1(\zeta_\ell, \alpha_\ell) \times U_3(\zeta_\ell, \alpha_\ell)$, then 
\begin{equation} \label{eqn:epsilon2} |z-\zeta| + |\alpha - \gamma|  <\delta  \ \ \text{ implies that } \ \  | G_j^{(\alpha_\ell, \zeta_\ell)}(\zeta, \alpha) - G_j^{(\alpha_\ell, \zeta_\ell)}(z, \gamma) | <\tfrac{\epsilon_0}{n}.\end{equation}
 Furthermore, the disjointness of the $U^j_2(\alpha_\ell, \zeta_\ell)$ implies that (after reordering with respect to the $j$ index) we must have 
\[ g_j^\alpha (\zeta) = G_j^{(\alpha_\ell, \zeta_\ell)}(\zeta,\alpha)  \ \text{ and } g_j^{\gamma}(z) = G_j^{(\alpha_\ell, \zeta_\ell)}(z,\gamma),\] 
for $j=1,\dots, n$. Then the desired inequality \eqref{eqn:epsilon} follows immediately from \eqref{eqn:epsilon2}.
\end{proof}

Now we can prove the general formula for Clark measures associated to exceptional values of two-variable RIFs. 

\begin{theorem} \label{thm:exceptional} Let $\phi=\frac{\tilde{p}}{p}$ be a bidegree $(m,n)$ RIF, and let $\alpha \in \mathbb{T}$ be exceptional for $\phi$. Then for $f\in C(\mathbb{T}^2)$, the associated Clark measure $\sigma_{\alpha}$ satisfies
\[\int_{\mathbb{T}^2}f(\zeta)d\sigma_{\alpha}(\zeta)=\sum_{j=1}^n\int_{\mathbb{T}}f(\zeta,g^\alpha_j(\zeta))\frac{dm(\zeta)}{|\frac{\partial \phi}{\partial z_2}(\zeta, g^\alpha_j(\zeta))|} + \sum_{k=1}^\ell c_k^{\alpha} \int_\mathbb{T} f(\tau_k, \zeta) dm(\zeta),\]
where $g^\alpha_1,\ldots, g^\alpha_n$ are the parametrizing functions from Lemma \ref{lem:levsets}, $\zeta_1 = \tau_1, \dots, \zeta_1=\tau_\ell$ are the vertical lines in $\mathcal{C}_{\alpha}(\phi)$ from Lemma \ref{lem:levsets}, and the $c^\alpha_k=|\frac{\partial \phi}{\partial z_1}(\tau_k,z_2)|^{-1}>0$ are constants.
\end{theorem}

\begin{proof} This proof uses many of the same arguments as the proof of Proposition 3.9 in \cite{BCSMich}, though it needs the additional tools of Lemma \ref{lem:excvalueconst}  and Lemma \ref{lem:unif}. For the ease of the reader, we still include the details below.

First, write $\mathcal{C}_\alpha(\phi)$ as $E_\alpha \cup (\cup_{k=1}^\ell L_k)$ where $E_\alpha$ is the set parametrized by the $g_j^\alpha$ and each $L_k$ denotes the vertical line $\{\zeta_1 = \tau_k\}$ in $\mathbb{T}^2$. As Clark measures do not have point-masses \cite[Theorem 2.1]{BCSMich}, $\sigma_\alpha(E_\alpha \cap L_k) =0$ for each $k$. Thus, it suffices for us to show
\begin{align}  \label{eqn:sigma1}
\int_{\mathbb{T}^2} f(\zeta) \chi_{E_\alpha}(\zeta) d\sigma_\alpha(\zeta) &=\sum_{j=1}^n\int_{\mathbb{T}}f(\zeta,g^\alpha_j(\zeta))\frac{dm(\zeta)}{|\frac{\partial \phi}{\partial z_2}(\zeta, g^\alpha_j(\zeta))|} \\
\label{eqn:sigma2}
\int_{\mathbb{T}^2} f(\zeta) \chi_{L_k}(\zeta) d\sigma_\alpha(\zeta) &= c_k^{\alpha} \int_\mathbb{T} f(\tau_k, \zeta) dm(\zeta) \end{align}
for all $f \in C(\mathbb{T}^2)$ and $k=1, \dots, \ell$, where we use $\chi_E$ to denote the characteristic function of a set $E$.   \\

\noindent \textbf{Part 1.} Let us first establish \eqref{eqn:sigma1}. Let $f \in C(\mathbb{T}^2)$. Fix $\epsilon >0$ sufficiently small and define $S_\epsilon$ and $S_{\epsilon/2}$ as in \eqref{eqn:Se} for $\hat{\epsilon} =\epsilon$ and $\hat{\epsilon } =\frac{\epsilon}{2}$ respectively. It is worth noting that the $\tau_1, \dots, \tau_\ell$ in the current proof  (the constant values for the lines $L_k$) form a subset of the $\tau_1, \dots, \tau_K$ from \eqref{eqn:Se} (the $z_1$-coordinates of the singularities of $\phi$ on $\mathbb{T}^2$). Thus, it makes sense to assume that the line-values appear at the beginning of the singularity-values list and then use $\tau_k$ to denote elements from either list.

Now, let $(\alpha_i) \subseteq \mathbb{T}$ be a sequence converging to $\alpha$ with  each $\alpha_i$ generic. By Corollary 2.2 in \cite{D19}, we know that the sequence $(\sigma_{\alpha_i})$ converges weak-$\star$ to $\sigma_\alpha$. To use that, let  $\Psi_\epsilon$ be a function in $C(\mathbb{T})$ satisfying
\[ \Psi_\epsilon \equiv 1 \text{ on } \mathbb{T} \setminus S_\epsilon, \ \ \Psi_\epsilon \equiv 0 \text{ on } S_{\epsilon/2}, \ \ 0 \le \Psi_\epsilon \le 1 \text{ on } S_\epsilon \setminus S_{\epsilon/2}.\]
By these assumptions and by Theorem \ref{thm:2clarkformula}, we have
\begin{align} 
\int_{\mathbb{T}^2} f(\zeta) \Psi_\epsilon(\zeta_1) d\sigma_\alpha(\zeta) &= \lim_{i \rightarrow \infty} 
\int_{\mathbb{T}^2} f(\zeta) \Psi_\epsilon(\zeta_1) d\sigma_{\alpha_i}(\zeta) \nonumber \\
& =\lim_{i \rightarrow \infty}  \sum_{j=1}^n\int_{\mathbb{T}}f(\zeta,g^{\alpha_i}_j(\zeta))  \Psi_\epsilon(\zeta) \frac{dm(\zeta)}{|\frac{\partial \phi}{\partial z_2}(\zeta, g^{\alpha_i}_j(\zeta))|} \nonumber  \\
& = \sum_{j=1}^n\int_{\mathbb{T}}f(\zeta,g^\alpha_j(\zeta))  \Psi_\epsilon(\zeta) \frac{dm(\zeta)}{|\frac{\partial \phi}{\partial z_2}(\zeta, g^\alpha_j(\zeta))|}, \label{eqn:sigma3}
\end{align}
where the last equality follows from Lemma \ref{lem:unif}, which implies that $F(\zeta, \alpha)$ as defined in \eqref{eqn:Fazeta} is uniformly continuous on $(\mathbb{T} \setminus S_{\epsilon/2}) \times \mathbb{T}$. Thus, $F( \zeta, \alpha) \Psi_\epsilon(\zeta) $ is uniformly continuous on $\mathbb{T}^2$. 
Since $\Psi_\epsilon(\zeta_1)\equiv 0$ on each line $L_k$, we can conclude that
\[ \begin{aligned}  \left| \int_{\mathbb{T}^2} f(\zeta) \chi_{E_\alpha}(\zeta) d \sigma_\alpha(\zeta) -  \int_{\mathbb{T}^2} f(\zeta) \Psi_\epsilon(\zeta_1) d \sigma_\alpha(\zeta) \right| &\le    \int_{\mathbb{T}^2} |f(\zeta)|(1-\Psi_{\epsilon}(\zeta_1)) \chi_{E_\alpha}(\zeta) d \sigma_\alpha(\zeta) \\
&\le  \| f \|_{L^{\infty}(\mathbb{T}^2)} \sigma_\alpha( (S_\epsilon \times \mathbb{T}) \cap E_\alpha).
\end{aligned}
\]
As $\epsilon \searrow 0$, the set $(S_\epsilon \times \mathbb{T}) \cap E_\alpha$ shrinks to a finite set of points and so, 
\[ \lim_{\epsilon \searrow 0}  \sigma_\alpha( (S_\epsilon \times \mathbb{T}) \cap E_\alpha) =0.\]
From this, we can conclude
\[   \int_{\mathbb{T}^2} f(\zeta) \chi_{E_\alpha}(\zeta) d \sigma_\alpha(\zeta) =\lim_{\epsilon \searrow 0}  \int_{\mathbb{T}^2} f(\zeta) \Psi_\epsilon(\zeta_1) d \sigma_\alpha(\zeta).\] 
Meanwhile, by breaking $f$ into its real and imaginary parts and then their positive and negative parts, we can use the monotone convergence theorem to conclude that 
\[ \lim_{\epsilon \searrow 0} \sum_{j=1}^n\int_{\mathbb{T}}f(\zeta,g^\alpha_j(\zeta))  \Psi_\epsilon(\zeta) \frac{dm(\zeta)}{|\frac{\partial \phi}{\partial z_2}(\zeta, g^\alpha_j(\zeta))|} = \sum_{j=1}^n\int_{\mathbb{T}}f(\zeta,g^\alpha_j(\zeta)) \frac{dm(\zeta)}{|\frac{\partial \phi}{\partial z_2}(\zeta, g^\alpha_j(\zeta))|}.\]
Combining these last two equalities with \eqref{eqn:sigma3} yields \eqref{eqn:sigma1}.\\

\noindent \textbf{Part 2.} To establish  \eqref{eqn:sigma2}, we follow the proof from \cite{BCSMich} and show that \eqref{eqn:sigma2} holds for all Poisson kernels $P_z$, where $z\in \mathbb{D}^2$. Then the result follows immediately, since linear combinations of these are dense in $C(\mathbb{T}^2)$.

First, fix 
 $r \in (0,1)$. The definition of $\sigma_\alpha$ implies that
\begin{equation} \label{eqn:sigma4}\int_{\mathbb{T}^2} P_{(r \tau_k,z_2)}(\zeta) d\sigma_\alpha(\zeta) =  \Re \left( \frac{\alpha + \phi(r \tau_k,z_2)}{\alpha-\phi(r \tau_k,z_2)}\right). \end{equation}
For the remainder of the proof, we basically just multiply both sides of \eqref{eqn:sigma4} by $(1-r)$ and take limits.
First, for $\zeta \in \mathbb{T}^2$, one can check that
\[ \lim_{r\rightarrow 1^{-}} (1-r) P_{(r\tau_k,z_2)}(\zeta) = \left\{\begin{array}{cc} 0 & \text{ if } \zeta_1 \ne \tau_k,\\
 2 P_{z_2}(\zeta_2) & \text{ if } \zeta_1 = \tau_k. \end{array} \right.\]
 Then the dominated convergence theorem implies that 
 \[ \lim_{r\rightarrow 1^-} \int_{\mathbb{T}^2} (1-r)  P_{(r \tau_k,z_2)}(\zeta) d\sigma_\alpha(\zeta) = \int_{\mathbb{T}^2} 2 P_{z_2} (\zeta_2) \chi_{L_k} (\zeta) d\sigma_{\alpha}(\zeta).\]
Since $L_k \subseteq \mathcal{C}_\alpha$, the maximum modulus principle implies that $\phi(\tau_k, z_2) = \alpha$ for all $z_2 \in \mathbb{D}$. Furthermore, since $\phi$ is analytic at each $(\tau_k,z_2)$ for $z_2 \in \mathbb{D}$, we have
 \[ \lim_{z_1 \rightarrow \tau_k} \phi(z_1, z_2) =\alpha \text{ and } \lim_{z_1 \rightarrow \tau_k} \frac{ \phi(z_1, z_2)-\alpha}{z_1 -\tau_k} = \tfrac{\partial \phi}{\partial z_1}(\tau_k , z_2) := d^\alpha_k \ne 0,\] 
 by Lemma \ref{lem:excvalueconst}. Then Carath\'eodory's theorem (for instance, consult (VI-3) in \cite{Sar94}) gives 
 \[  \lim_{r\rightarrow 1^-} \frac{ 1-|\phi(r \tau_k,z_2)|}{1-r} = d_k^{\alpha} \tau_k \bar{\alpha} = |d_k^{\alpha}|\] 
 and so
 \[
 \begin{aligned}
 \lim_{r\rightarrow 1^-} \Re\left( \frac{(1-r) (\alpha +\phi(r \tau_k, z_2))}{ \alpha-\phi(r \tau_k, z_2)} \right) &=
 \lim_{r\rightarrow 1^-}
(1-r) \frac{ 1-|\phi(r\tau_k,z_2)|^2}{|\alpha - \phi(r\tau_k,z_2)|^2}  \\
& = \lim_{r\rightarrow 1^-} 2 \left | \frac{\tau_k-r\tau_k}{\alpha - \phi(r \tau_k,z_2)}\right|^2  \frac{ 1-|\phi(r \tau_k,z_2)|}{1-r}  = \frac{2}{|d^\alpha_k|}.
\end{aligned}
\]
To finish the proof, set
\begin{equation} \label{eqn:formc} c^\alpha_k=\frac{1}{|d^\alpha_k|} = \frac{1}{| \tfrac{\partial \phi}{\partial z_1}(\tau_k , z_2)|} >0.\end{equation}
 Then \eqref{eqn:sigma4} paired with our prior computations imply that
 \[   \int_{\mathbb{T}^2} P_{z_2} (\zeta_2) \chi_{L_k} (\zeta) d\sigma_{\alpha}(\zeta) = c^\alpha_k =c^\alpha_k \int_{\mathbb{T}} P_{z_2}(\zeta) dm(\zeta).\]
If we multiply both sides by $P_{z_1}(\tau_k)$, this gives \eqref{eqn:sigma2} for $f=P_{z_1}P_{z_2}$, which is what we were trying to show.
\end{proof}

\section{Analysis of Clark embeddings}\label{sec:iso}

Let $\phi$ be a two-variable RIF and recall that $K_\phi$ denotes the two-variable model space associated to $\phi$. Then, as discussed earlier,  Doubtsov in \cite{D19} studied the canonical isometry $J_\alpha: K_\phi \mapsto L^2(\sigma_\alpha)$, which is initially defined on reproducing kernels by
\[
J_\alpha[K_w](\zeta) := (1- \alpha \overline{\phi(w)})C_w(\zeta), \quad \text{for } w \in \mathbb{D}^d, \zeta \in \mathbb{T}^d
\]
and then extended to all functions in $K_\phi$.  In this section, we characterize when $J_\alpha$ is unitary. As in the previous section, we must consider the cases of generic $\alpha$ and exceptional $\alpha$ separately. 

\subsection{Clark embeddings associated to generic values}
Our structure theorem for Clark measures associated with generic $\alpha\in \mathbb{T}$ allows us to show that the corresponding Clark embedding operators are surjective. We achieve this by  showing that, for $\alpha \in \mathbb{T}$ generic, the bidisk algebra $A(\mathbb{D}^2)$ is dense in $L^2(\sigma_{\alpha})$. Then we can appeal to Doubtsov's result \cite[Theorem 3.2]{D19}. Note that we are excluding the degenerate case when $\phi$ is a function of one variable only; in that case $J_{\alpha}$ fails to be unitary for all $\alpha\in \mathbb{T}$, cf. \cite{D19}. 

We first need the following auxiliary lemma.

\begin{lemma} \label{lem:trigpoly} Let $\phi=\frac{\tilde{p}}{p}$ be a bidegree $(m,n)$ RIF  and suppose
 $\alpha\in \mathbb{T}$ is a generic value for $\phi$. Then
there exist rational functions $R_1,R_2\in A(\mathbb{D}^2)$ such that 
$\bar{z}_1=R_1$ and $\bar{z}_2=R_2$
on $\mathcal{C}_{\alpha}(\phi)$.
\end{lemma}
\begin{proof}
Since neither $\tilde{p}$ nor $p$ are polynomials in one variable only, we may write
\[\phi(z)=\frac{q_1(z_2)+z_1q_2(z_1,z_2)}{p_1(z_2)+z_1p_2(z_1,z_2)}\]
for some polynomials $p_1,q_1\in \mathbb{C}[z_2]$ and $p_2,q_2\in\mathbb{C}[z_1,z_2]$.  
Similarly, we have
\[\phi(z)=\frac{r_1(z_1)+z_2r_2(z_1,z_2)}{s_1(z_1)+z_2s_2(z_1,z_2)}\]
again with $r_1,s_1$ being polynomials in one variable, and $r_2,s_2$ being polynomials in $z_1,z_2$.
Using the first representation, we can rewrite the expression $\tilde{p}(z)-\alpha p(z)$ as
\begin{align}  \nonumber
\tilde{p}(z)-\alpha p(z) &=  q_1(z_2)+z_1q_2(z_1,z_2)-\alpha(p_1(z_2)+z_1p_2(z_1,z_2))\\
&=q_1(z_2)-\alpha p_1(z_2)+z_1(q_2(z_1,z_2)-\alpha p_2(z_1,z_2)). \label{levcurverewrite}
\end{align}
Suppose now that $q_1(z_2)-\alpha p_1(z_2)$ does not vanish on $\overline{\mathbb{D}}$. Then \eqref{levcurverewrite} implies that $\tilde{p}(z)-\alpha p(z)=0$ precisely when 
\[1=\frac{z_1(\alpha p_2(z_1,z_2)-q_2(z_1,z_2))}{q_1(z_2)-\alpha p_1(z_2)}\]
and if $\zeta \in \mathcal{C}_{\alpha}$, we can rewrite this as
\[\bar{\zeta_1}=\frac{\alpha p_2(\zeta_1,\zeta_2)-q_2(\zeta_1,\zeta_2)}{q_1(\zeta_2)-\alpha p_1(\zeta_2)}.\]
Since its denominator is non-vanishing in the closed unit disk, the rational function on the right belongs to the bidisk algebra $A(\mathbb{D}^2)$. Similarly, if $r_1(z_1)-\alpha s_1(z_2)\neq 0$ for $z_2\in \overline{\mathbb{D}}$, then $\bar{z}_2$ can be seen to be equal to a rational function in $A(\mathbb{D}^2)$ on $\mathcal{C}_{\alpha}$.
If both $q_1-\alpha p_1$ and $r_1-\alpha s_1$ are non-vanishing on the closed unit disk, the assertion of the lemma follows. 

It remains to prove that the assumption that $\alpha$ is generic rules out the presence of zeros in $\overline{\mathbb{D}}$. First note that any zero of $q_1(z_2)-\alpha p_1(z_2)$ in $\overline{\mathbb{D}}$ must in fact belong to $\mathbb{T}$; otherwise, $\phi$ would be unimodular at $(0,z_2)$, which is impossible since $\phi$ is a nonconstant RIF. Seeking a contradiction, we assume $q_1(\tau)=\alpha p_1(\tau)$ for some $\tau \in \mathbb{T}$. Consider the function
\[\phi_{\tau}(z_1)=\phi(z_1,\tau)=\frac{q_1(\tau)+z_1q_2(z_1,\tau)}{p_1(\tau)+z_1p_2(z_1,\tau)},\]
and note that $\phi_{\tau}$ is a finite Blaschke product. Evaluating at $z_1=0$, we obtain that 
\[\phi_{\tau}(0)=\frac{q_1(\tau)}{p_1(\tau)}=\alpha\frac{p_1(\tau)}{p_1(\tau)}=\alpha.\]
Hence $\phi_{\tau}(z_1)\equiv \alpha$ for $z_1\in \overline{\mathbb{D}}$. But this amounts to saying that $\mathbb{T}\times \{\tau\}\subset \mathcal{C}_{\alpha}$ which contradicts the assumption that $\alpha$ is generic. A similar argument applies to $r_1-\alpha s_1$, and the proof is complete.
\end{proof}

Now, we can show that the Clark embedding operators corresponding to generic values are surjective.

\begin{theorem} Let $\phi=\frac{\tilde{p}}{p}$ be a bidegree $(m,n)$ RIF
and let $\alpha \in \mathbb{T}$ be a generic value for $\phi$. Then the Clark embedding $J_{\alpha}\colon K_{\phi}\to L^2(\sigma_{\alpha})$ is unitary.
\end{theorem}
\begin{proof}  Using Theorem 3.2 in \cite{D19}, it will suffice to show that $A(\mathbb{D}^2)$ is dense in $L^2(\sigma_{\alpha})$. 
Since $\sigma_{\alpha}$ is a finite Borel measure on the compact set $\mathbb{T}^2$ (and hence, a finite Radon measure), $C(\mathbb{T}^2)$ is dense in  $L^2(\sigma_{\alpha})$. Moreover, by the Stone-Weierstrass theorem, the set of two-variable trigonometric polynomials is dense in $C(\mathbb{T}^2)$. From this, it is easy to show that the set of two-variable trigonometric polynomials is also dense in $L^2(\sigma_{\alpha})$. Specifically, fix $f \in L^2(\sigma_{\alpha})$, 
and $\epsilon >0$. Then there is a function $g \in C(\mathbb{T}^2)$ and a trigonometric polynomial $p$ such that 
\[ \|f - g\|_{L^2(\sigma_\alpha)} < \tfrac{\epsilon}{2} \ \ \text{ and } \ \ \max_{z \in \mathbb{T}^2} | g(z) - p(z) | < \frac{\epsilon}{2 \sigma_\alpha(\mathbb{T}^2)}.\]
Then 
\[ \|f - p\|_{L^2(\sigma_\alpha)} \le \|f - g\|_{L^2(\sigma_\alpha)}  + \|g - p \|_{L^2(\sigma_\alpha)} <\epsilon,\]
as needed. By Lemma \ref{lem:trigpoly}, we know that each trigonometric polynomial agrees with some function in $A(\mathbb{D}^2)$ on $\mathcal{C}_\alpha(\phi)$, which contains the support of $\sigma_\alpha$. Thus,  $A(\mathbb{D}^2)$ is dense in $L^2(\sigma_{\alpha})$.
\end{proof}

\subsection{Clark embeddings associated to exceptional values}

Let us now consider the case of exceptional values. In this situation, the Clark embedding operators are never surjective.

\begin{theorem}
Let $\phi=\frac{\tilde{p}}{p}$ be a bidegree $(m,n)$ RIF
and let $\alpha \in \mathbb{T}$ be an exceptional value for $\phi$. Then the Clark embedding $J_{\alpha}\colon K_{\phi}\to L^2(\sigma_{\alpha})$ is not unitary.
\end{theorem}

\begin{proof}
The proof proceeds along the same lines as the second half of the proof of \cite[Proposition 3.10]{BCSMich}.  Namely, by Theorem \ref{thm:exceptional}, for $f \in C(\mathbb{T}^2)$, among the nonnegative terms that make up 
\[ \int_{\mathbb{T}^2}|f (\zeta)|^2d\sigma_{\alpha}(\zeta)\] 
is a term of one of the following two forms:
\[ c_1\int_{\mathbb{T}}|f(\tau, \zeta)|^2dm(\zeta) \text{ or }
c_1\int_{\mathbb{T}}|f(\zeta, \tau)|^2dm(\zeta),\]
for some $c_1 \ne 0$.
 Without loss of generality, we assume the former. By Theorem 3.2 in \cite{D19},  $J_{\alpha}$ is unitary if and only if $A(\mathbb{D}^2)$ is dense in $L^2(\sigma_{\alpha})$. Now for any choice of $f\in A(\mathbb{D}^2)$, the function $f(\tau, \zeta_2)$ belongs to $H^2(\mathbb{T})$. But $H^2(\mathbb{T})$ has positive distance from the $L^2(\mathbb{T})$-span of the function $g(\zeta)=\bar{\zeta}_2$, which is an element of $L^2(\sigma_{\alpha})$ since $g$ is continuous and $\sigma_{\alpha}$ is Radon. 
Hence $A(\mathbb{D}^2)$  fails to be dense in $L^2(\sigma_{\alpha})$, and the assertion follows.
\end{proof}

\section{Fine structure of Clark measures for two-variable RIFs}\label{sec:structtwo}
In this section, we continue our study of Clark measures associated with RIFs in the bidisk and our main goal is to use results from \cite{BKPS} to address a question raised in \cite[Remark 5.5]{BCSMich}. Specifically, for a bidegree $(n,1)$ RIF $\phi,$ it was observed that the Clark measures associated to generic $\alpha$ values exhibited a certain type of vanishing at the singularities of $\phi$ and that this vanishing appeared connected to the notion of contact order from \cite{BPS17}. In what follows, we will make these ideas precise and prove an order of vanishing result in the more general bidegree $(m,n)$ context of this paper. 

\subsection{Preliminaries on fine structure.}
First, let   $\phi=\frac{\tilde{p}}{p}$ be a bidegree $(m,n)$ RIF, and let $\alpha \in \mathbb{T}$ be generic. With the notation from Theorem \ref{thm:2clarkformula} and for $j=1, \dots, n$, let $W^\alpha_j$ denote the  the weight function
\begin{equation}
W^{\alpha}_j(\zeta)=\left|\frac{\partial \phi}{\partial z_2}\left(\zeta, g^{\alpha}_j(\zeta)\right)\right|^{-1}
\label{eq:Wadef}
\end{equation}
that appears in the Clark measure formula. In what follows, it will be useful to have the additional information about $W^\alpha_j$ encoded in the following lemma.

\begin{lemma}\label{lem:Wreduce} Let $\phi=\frac{\tilde{p}}{p}$ be a bidegree $(m,n)$ RIF
and  let $\alpha\in \mathbb{T}$ is generic. Then for each $j$, the function $W^{\alpha}_j \in L^1(\mathbb{T})$ and  satisfies the formula
\begin{equation}
W^{\alpha}_j(\zeta)=\frac{|p(\zeta, g^{\alpha}_j(\zeta)|}{\left|\frac{\partial \tilde{p}}{\partial z_2}(\zeta, g^{\alpha}_j(\zeta))-\alpha \frac{\partial p}{\partial z_2}(\zeta, g^{\alpha}_j(\zeta))\right|}.
\label{eq:Wareduct}
\end{equation}
\end{lemma}
\begin{proof} Fix a generic $\alpha \in \mathbb{T}$. Observe that since
\[ \sum_{j=1}^n\int_{\mathbb{T}} W^{\alpha}_j(\zeta) dm(\zeta) = \int_{\mathbb{T}^2} d\sigma_\alpha(\zeta) = \frac{1-|\phi(0)|^2}{|\alpha-\phi(0)|^2}<\infty, \]
 each $W^{\alpha}_j$ must be in $L^1(\mathbb{T}).$ 
Next, note that
\[\frac{\partial \phi}{\partial z_2}(z_1,z_2)= \frac{\partial}{\partial z_2}\left(\frac{\tilde{p}}{p}\right)(z_1,z_2)=\frac{\frac{\partial \tilde{p}}{\partial z_2}(z_1,z_2)\cdot p(z_1, z_2)-\frac{\partial p}{\partial z_2}(z_1,z_2)\cdot \tilde{p}(z_1,z_2)}{p(z_1,z_2)^2}.\]
Now, by the definition of $\mathcal{C}_{\alpha}(\phi)$, we have $\tilde{p}(\zeta, g^{\alpha}_j(\zeta))=\alpha p(\zeta, g^{\alpha}_j(\zeta))$, meaning we can cancel a common factor of 
$p(\zeta, g^{\alpha}_j(\zeta))$ to obtain
\[\frac{\partial \phi}{\partial z_2}(\zeta, g^{\alpha}_j(\zeta))=\frac{\frac{\partial \tilde{p}}{\partial z_2}(\zeta, g^{\alpha}_j(\zeta))-\alpha \frac{\partial p}{\partial z_2}(\zeta, g^{\alpha}_j(\zeta))}{p(\zeta, g^{\alpha}_j(\zeta))}.\]
Taking reciprocals and moduli, we obtain the desired formula.
\end{proof}

Now let $(\tau, \gamma)\in \mathbb{T}^2$ be a singularity of $\phi$. This implies $\alpha p(\tau, \gamma)=0=\tilde{p}(\tau, \gamma)$ and so, $(\tau, \gamma) \in \mathcal{C}_{\alpha}(\phi)$ for each $\alpha$. Then, Lemma \ref{lem:Wreduce} suggests that each $W^{\alpha}_j(\zeta)$ should probably vanish at $\zeta=\tau$, provided $(\tau, g^{\alpha}_j(\tau))=(\tau, \gamma)$. Note, however, that  not every component of $\mathcal{C}_{\alpha}$ need go through each singularity, i.e. satisfy $g^{\alpha}_j(\tau) =\gamma$. This is exhibited in Figure \ref{fig:avoidlc},  which is reproduced from \cite[Example 7.2]{BPSprep} and displays some level curve components that do not go through the singularity at $(1,1)$. 
Thus 
$|p(\zeta, g^{\alpha}_j(\zeta))|$ may be strictly positive for some indices $j$.  

\begin{figure}
\includegraphics[width=0.4 \textwidth]{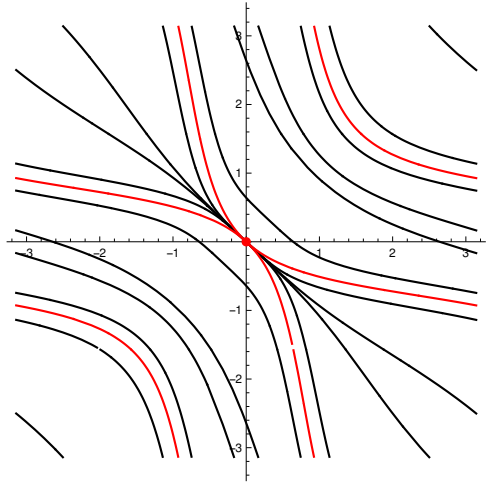}
\caption{Level curves for a RIF $\phi$ with a level curve component $(\zeta, g^{\alpha}_j(\zeta))$ that does not pass through the singularity at $(1,1)$, graphed via arguments on $[-\pi, \pi)^2.$} \label{fig:avoidlc}
\end{figure}

The goal for the remainder of this section is to prove the following statement:

\begin{center}
  \emph{For most values of $\alpha$, if a branch $(\zeta,  g^{\alpha}_j(\zeta))$ of $\mathcal{C}_{\alpha}$ goes through the singularity $(\tau, \gamma)$, then the corresponding weight function $W^{\alpha}_j$ has order of vanishing at $\tau$ that corresponds to the ``contact order'' of the corresponding branch of $\mathcal{Z}(\tilde{p})$ at the point $(\tau, \gamma)$.} 
\end{center}

\subsection{Change of variables and key definitions}

To make notions of contact order precise and connect them to level set behaviors, we switch to the setting of the upper half-plane and use the machinery developed in \cite{BKPS}. This will flatten the distinguished boundary $\mathbb{T}^2$ to $\mathbb{R}^2$ and allow us to move the singularity to the origin $(0,0)$. While we present the salient details of the change of variables here, the interested reader can find additional details and results in Section $2$ of  \cite{BKPS}. 

First, without loss of generality, assume that the singularity of interest is $(\tau, \gamma)=(1,1)$. Let $\mathbb{H}$ denote the upper half-plane $\mathbb{H}=\{z\in \mathbb{C}\colon \mathrm{Im}(z)>0\}$ and let $\beta \colon \mathbb{D}\to \mathbb{H}$ and $\beta^{-1}\colon \mathbb{H}\to \mathbb{D}$ be the conformal maps
\[\beta(z)=i\left(\frac{1-z}{1+z}\right) \quad \textrm{and} \quad \beta^{-1}(z)=\frac{1+iz}{1-iz}.\]
Recall that $\phi=\frac{\tilde{p}}{p}$ is a bidegree $(m,n)$ RIF. To convert $\phi$ to $\mathbb{H}$, first define a new polynomial $q$ by
\[ q(z) = (1-iz_1)^{m} (1-iz_2)^{n} p\left(\beta^{-1}(z_1), \beta^{-1}(z_2)\right).\]
Then $q$ is a polynomial with no zeros on $\mathbb{H}^2$ but a zero at $(0,0)$. Similarly, define
\[ \bar{q}(z) := (1-iz_1)^{m} (1-iz_2)^{n} \tilde{p}\left(\beta^{-1}(z_1), \beta^{-1}(z_2)\right).\]
One can check that 
\[ \bar{q}(z)=  \overline{q(\bar{z}_1, \bar{z}_2)}\]
and the rational function
 \[\Psi(z):=\frac{\bar{q}(z)}{q(z)}\] 
satisfies $|\Psi(z_1,z_2)| \le 1$ on $\mathbb{H}^2$ and $ |\Psi(z_1,z_2)| =1$ a.e.~on $\mathbb{R}^2$ with a singularity at $(0,0)$. Thus, we have transformed $\phi$ on $\mathbb{D}^2$ with a singularity at $(1,1)$ to $\Psi$ on $\mathbb{H}^2$ with a singularity at $(0,0)$.

Now we need notation to identify level sets in this context. Specifically, for $\alpha \in \mathbb{T}$, define the set
\[ \mathcal{V}_{\alpha}(\Psi) =\left\{x \in \mathbb{R}^2: \bar{q}(x)-\alpha q(x)=0 \right\}.\]
Let  $\zeta_2 = g^{\alpha}_j(\zeta_1)$ be a branch of $\mathcal{C}_{\alpha}(\phi)$. Then if we define $h^{\alpha}_j  := \beta \circ g^{\alpha}_j \circ \beta^{-1}$, the curve $x_2 = h^{\alpha}_j(x_1)$ must be  a branch of $\mathcal{V}_{\alpha}(\Psi)$ and this fact is clearly reversible. This is useful because Theorems $2.16$ and $2.20$ in \cite{BKPS} give information connecting the branches of the zero set $\mathcal{Z}(q)$ and the branches of the level set $ \mathcal{V}_{\alpha}(\Psi)$ via their respective Puiseux expansions. (See \cite[Chapter 7]{FisBook} for a detailed overview of how to parametrize branches of an algebraic curve using Puiseux series, and \cite[Section 2]{BKPS} for a fuller discussion of the specific forms these take in the present setting.)

We encode the important information from \cite{BKPS} in the following two theorems:

\begin{theorem}  \label{thm:qfactor}There is a positive integer $J$ and related positive integers $M_1, \dots, M_J$ such that near $(0,0)$, $q$ factors as
\begin{equation} \label{eqn:qfactor} q(z)=u(z)\prod_{j=1}^J\prod_{m=1}^{M_j}\left(z_2+q_j(z_1)+z_1^{2L_j}\psi_j(\mu^{m}_jz_1^{\frac{1}{M_j}})\right),\end{equation}
where $u$ is a unit at $(0,0)$ ( i.e. $u$ is analytic with $u(0,0)\neq 0$), each $q_j\in \mathbb{R}[z]$ (i.e.~ is a polynomial with real coefficients) has $\deg q_j <2L_j$, $q_j(0)=0$, and $q'_j(0)>0$, each $\mu_j = \text{exp}(2i\pi/M_j)$ is a primitive root of unity, 
 and each $\psi_j$ is analytic near the origin with $\mathrm{Im}(\psi_j(0))>0$. 
\end{theorem}

\begin{theorem} \label{thm:Valpha} For all but finitely many $\alpha \in \mathbb{T}$, we can also factor
\begin{equation}
\bar{q}(z)-\alpha q(z)=(1-\alpha) u^{\alpha}(z)\prod_{j=1}^J\prod_{m=1}^{M_j}\left(z_2+q_j(z_1)+z_1^{2L_j}\psi^{\alpha}_{j,m}(z_1)\right),
\label{eq:BKPSformula}
\end{equation}
where $u^{\alpha}$ is a unit at $(0,0)$, each $q_j$ is the same as in \eqref{eqn:qfactor} and each $\psi^{\alpha}_{j,m}$ is a real analytic function in a neighborhood of the origin. \end{theorem} 

\begin{remark} Basically, Theorems \ref{thm:qfactor} and \ref{thm:Valpha} describe the branches of the zero set $\mathcal{Z}(q)$ and the level set $ \mathcal{V}_{\alpha}(\Psi)$ near $(0,0)$. To denote that we are now looking specifically at $\mathbb{R}^2$, we use variables $x = (x_1, x_2)$. Then Theorem \ref{thm:qfactor} says that near $(0,0)$, each branch of  $\mathcal{Z}(q)$ in $\mathbb{R}^2$ is of the form
\[ x_2 = -q_j(x_1)-x_1^{2L_j}\psi_j\left(\mu^{m}_jx_1^{\frac{1}{M_j}}\right),\]
where the $q_j$ and $\psi_j$ satisfy certain properties. Similarly, Theorem \ref{thm:Valpha} says that near $(0,0)$, each branch of $ \mathcal{V}_{\alpha}(\Psi)$ is of the form
\[ x_2 = -q_j(x_1)-x_1^{2L_j}\psi^{\alpha}_{j,m}(x_1),\]
where the $q_j$ and $\psi^{\alpha}_{j,m}$ satisfy certain properties. It is very important to note that the polynomials $q_j$ are the same in the two theorems. However, the $\psi_j$ and $\psi^{\alpha}_{j,m}$ are not. They actually exhibit immediate disagreement because $\text{Im}(\psi_j(0)) \ne 0$ and $\psi^{\alpha}_{j,m}(0) \in \mathbb{R}$.
\end{remark}

We let $\mathbb{T}_q$ denote the set of $\alpha \in \mathbb{T}$ for which Theorem \ref{thm:Valpha} applies. This is $\mathbb{T}$ with a finite set removed and corresponds to the $\alpha$ for which $\bar{q} -\alpha q$ has a factorization mirroring that of $q$. As a consequence, for $\alpha \in \mathbb{T}_q$, can define the contact order of each branch of $\mathcal{C}_\alpha$ as follows.

\begin{definition} \label{def:contactorder} Let $\zeta_2 = g^{\alpha}_j(\zeta_1)$ be a branch of $\mathcal{C}_{\alpha}$ going through $(1,1)$ and let $x_2 = h^{\alpha}_j(x_1)$ be the corresponding branch of $\mathcal{V}_{\alpha}$ going through $(0,0)$. Then (referring to Theorem \ref{thm:Valpha}),
\[ h^{\alpha}_j(x_1) = -q_k(x_1)-x_1^{2L_k}\psi^{\alpha}_{k,m}(x_1), \text{ for some pair $(k,m)$}.\]
We then say that $K_j:=2L_k$ is both the \textbf{contact order} of $\phi$  at $(1,1)$ for the branch  $\zeta_2 = g^{\alpha}_j(\zeta_1)$ of $\mathcal{C}_{\alpha}$ and the \textbf{contact order} of $\Psi$ at $(0,0)$ for the branch $x_2 = h^{\alpha}_j(x_1)$ of $\mathcal{V}_{\alpha}$.
\end{definition} 

We now have enough machinery to state our main result, which precisely describes how a weight function $W^\alpha_j$ behaves near a singularity $(\tau, \gamma)$ of $\phi$.

\begin{theorem} \label{thm:Waj}
Assume the setup of Theorem \ref{thm:2clarkformula} and let $W^\alpha_j$ be given as in \eqref{eq:Wadef}.  For all but finitely many $\alpha \in \mathbb{T}$, the following holds. 
If $(\tau, \gamma) \in\mathbb{T}^2$ is a singularity of $\phi$ and $\zeta_2 = g^{\alpha}_j(\zeta_1)$ is a branch of $\mathcal{C}_\alpha$ going through $(\tau, \gamma)$, then there are constants $c, C$ such that
\begin{equation} \label{eqn:walphabd} 0<c\leq \frac{W^{\alpha}_j(\zeta)}{|\zeta-\tau|^{K_j}}\leq C\end{equation}
for all $\zeta$ in a neighborhood of $\tau$, where $K_j$ is the contact order of $\phi$ at $(\tau,\gamma)$ associated with the branch $\zeta_2 = g^{\alpha}_j(\zeta_1)$ as given in Definition \ref{def:contactorder}.
\end{theorem}

The proof requires an additional technical lemma and so, we postpone the proof until the next section. However, we do observe a quick corollary.

\begin{corollary} Assume the setup of Theorem \ref{thm:2clarkformula} and let $W^\alpha_j$ be given as in \eqref{eq:Wadef}. If $\alpha$ is a parameter value for which Theorem \ref{thm:Waj} applies, then $W^{\alpha}_j$ is a bounded function on $\mathbb{T}$. 
\end{corollary}

\begin{proof} Let $\tau_1, \dots, \tau_k$ denote the $z_1$-coordinates of the singularities of $\phi$ that the branch $\zeta_2 = g^{\alpha}_j(\zeta_1)$ passes through. By Theorem \ref{thm:Waj}, there are open intervals $I_\ell$ around each $\tau_\ell$ such that $W^{\alpha}_j$ is bounded on $I_{\tau_\ell}$. Now, consider $W^{\alpha}_j$ on $\mathbb{T} \setminus \left( \cup_{\ell =1}^k I_\ell\right)$. By assumption,
\[ \left|\tfrac{\partial \phi}{\partial z_2}\left(\zeta, g^{\alpha}_j(\zeta)\right)\right|\]
is continuous on $\mathbb{T} \setminus \left( \cup_{\ell =1}^k I_\ell\right)$. Thus, if $W^{\alpha}_j$ is unbounded on $\mathbb{T} \setminus \left( \cup_{\ell =1}^k I_\ell\right)$, there must be a point $\tau_0$ such that 
\[  \left|\tfrac{\partial \phi}{\partial z_2}\left(\tau_0, g^{\alpha}_j(\tau_0)\right)\right| =0.\]
But, this would imply that the finite Blaschke product $\phi_{\tau_0}(z) := \phi(\tau_0, z)$ is constant and thus, the line $\{ \zeta^2 \in \mathbb{T}: \zeta_1 = \tau_0\}$ is in some $\mathcal{C}_\lambda$. Note that $\alpha$ is generic and so, is not equal to this $\lambda$. But, this implies that the point $(\tau_0, g^{\alpha}_j(\tau_0))$ is on two different level sets of $\phi$ and so, has to be a singularity of $\phi$. This is a contradiction and so, $W^{\alpha}_j$ must be bounded after all. 
\end{proof}

\subsection{Proof of Theorem \ref{thm:Waj}}

The proof of Theorem \ref{thm:Waj} basically involves translating $W^\alpha_j$ to the setting of $q, \bar{q}$ and using the factorization results to identify the natural order vanishing of the numerator and denominator of the translated $W^\alpha_j$ near the singularity. However, there is the possibility of additional, unexpected vanishing in the denominator. 

To account for that, we require a somewhat technical lemma that is based on the ideas from \cite{BKPS}. Specifically, we say that the branch of $q$ (as given in the factorization \eqref{eqn:qfactor}) with index $(j,m)$ has initial segment $r \in \mathbb{R}[z]$ of order $n$ if 
\[r(z_1) -\left( q_j(z_1)+z_1^{2L_j}\psi_j(\mu^{m}_jz_1^{\frac{1}{M_j}}) \right ) = O\left( |z_1|^n \right),\]
and for $\alpha \in \mathbb{T}_q$, we say  the branch of $\bar{q}-\alpha q$ (as given in the factorization \eqref{eq:BKPSformula}) with index $(j,m)$ has initial segment $r \in \mathbb{R}[z]$ of order $n$ if
\[r(z_1) -\left( q_j(z_1)+z_1^{2L_j}\psi^{\alpha}_{j,m}(z_1)\right ) = O\left( |z_1|^n \right).\]

Then we have the following lemma.

\begin{lemma} \label{lem:badalpha} Given the factorizations and definitions above, there are at most finitely many $\alpha \in \mathbb{T}$ such that for some pair $(j,m)$, 
\begin{equation} \label{eqn:rbdef} r(z_1):=q_j(z_1) + z_1^{2L_j}\psi^{\alpha}_{j,m}(0),\end{equation}
is an initial segment of a branch of $q$ of order $2 L_j +1$.
\end{lemma}

\begin{proof} 
Fix an index $j_0$ with $ 1\le j_0\le J$ and observe there are at most finitely many $b \in \mathbb{R}$ such that 
\begin{equation} \label{eqn:rbdef} r_b(z_1):=q_{j_0}(z_1) + bz_1^{2L_{j_0}}\end{equation}
is an initial segment of a branch of $q$ of order $2L_{j_0}+1$. In particular, if that happened, \eqref{eqn:rbdef} would have to be the initial part of a different $q_j$ appearing in the factorization of $q$. Since there are only finitely many such $q_j$, there are only finitely many such $b$.  In such situations, the given branch of $q$ (say with index $(j,m)$ as in \eqref{eqn:qfactor}) would have to satisfy $2L_j > 2L_{j_0}$ because 
$\text{Im} (\psi_j(0))  \ne 0$ and so, the degree $z_1^{2L_j}$ term in the branch cannot agree with the degree $z_1^{2L_j}$ term in $r_b$.

Fix a $(b, j_0)$ combination such that $r_b$ is an initial segment of order $2L_{j_0}+1$ of a branch of $q$. To prove the lemma, it will suffice to show that there is at most one $\alpha \in \mathbb{T}_q$ with $\psi^{\alpha}_{j_0,m}(0) =b$ for some $m$.  By assumption, there is some positive number $N$ such that  $r_b$ is an initial segment of order  $2L_{j_0}+1$ for exactly $N$ branches of $q$ as given in  \eqref{eqn:qfactor}. As this agreement must be happening between the terms in $r_b$ and the terms in the real polynomials $q_j$ in the branches, we can translate this information over to all $\bar{q}-\alpha q$ with $\alpha \in \mathbb{T}_q$. 

Specifically, for $\alpha \in \mathbb{T}_q$, the $q_j$ are the same in the two branch factorizations \eqref{eqn:qfactor} and \eqref{eq:BKPSformula}. Thus, $r_b$ also agrees to order $2L_{j_0}+1$ with $N$ of the $q_j$ (counted according to multiplicity) appearing in \eqref{eq:BKPSformula} and thus, is an initial segment of order $2L_{j_0}+1$  for at least $N$ branches of $\bar{q}-\alpha q$.  

Proceeding towards a contradiction, assume that two of these $\alpha$, call them $\alpha_1$ and $\alpha_2$, have $ \psi^{\alpha_i}_{j_0,m_i}(0) = b$ for some indices $m_1$ and $m_2$. Then $r_b$ is an initial segment of order $2L_{j_0}+1$ of the $(j_0, m_1)$ and $(j_0,m_2)$ branches of $\bar{q}-\alpha_1 q$ and $\bar{q}- \alpha_2 q$ respectively and these new branches are in addition to the $N$ branches already identified, since those had to satisfy $2L_j > 2L_{j_0}$.
This means that  $r_b$ is an initial segment of order $2L_{j_0}+1$ for at least $N+1$ branches of both $\bar{q}-\alpha_1 q$ and $\bar{q}-\alpha_2 q$.  But, by the discussion in the proof of Theorem 2.21 in  \cite{BKPS}, with the exception of at most one $\alpha \in \mathbb{T}_q$, $r_b$ must be an initial segment of order $2L_{j_0}+1$  for the same number of branches of $q$ and $\bar{q}-\alpha q$. This means that $r_b$ must be an initial segment of order $2L_{j_0}+1$  for at least $N+1$ branches of $q$, which gives our contradiction.
Thus, there is at most one $\alpha \in \mathbb{T}_q$ with $\psi^{\alpha}_{j_0,m}(0) =b$ for some $m$, and the proof is complete.
\end{proof}

Given that key technical lemma, we can now prove Theorem \ref{thm:Waj}.

\begin{proof} Without loss of generality, assume $(\tau, \gamma) =(1,1)$. Fix $\alpha \in \mathbb{T}$ and by omitting at most a finite number of $\alpha$, one can assume that $\alpha \in \mathbb{T}_q$ so the factorization in Theorem \ref{thm:Valpha} applies and $\alpha$ does not possess the behavior detailed in the statement of Lemma \ref{lem:badalpha}.

Let $\zeta_2 = g^{\alpha}_j(\zeta_1)$ be the branch $\mathcal{C}_\alpha$ going through $(1,1)$ associated with $W^\alpha_j$ and let $x_2 = h^{\alpha}_j(x_1)$ be the corresponding branch of $\mathcal{V}_\alpha$ going through $(0,0)$.
Define the related function 
\[V^{\alpha}_j(x)=\frac{|q(x,h_j^{\alpha}(x))|}{|\frac{\partial \bar{q}}{\partial z_2}(x,h_j^{\alpha}(x))-\alpha \frac{\partial q}{\partial z_2}(x,h_j^{\alpha}(x))|}.\]
Because $(\beta^{-1})'(x)$ is bounded above and below in a neighborhood of the origin, one can use the formula for $W^\alpha_j$ in \eqref{eq:Wareduct} to show that there are constants $d, D$ such that
\begin{equation} \label{eqn:valphabd} 0<d \leq \frac{V^{\alpha}_j(x)}{|x|^{K_j}}\leq D\end{equation}
if and only if \eqref{eqn:walphabd} holds. The remainder of the proof establishes \eqref{eqn:valphabd} by identifying the order of vanishing at $x=0$ of both the numerator and denominator of $V^\alpha_j$ and showing that the difference in these orders of vanishing is exactly $K_j$.

We first study the denominator of $V^{\alpha}_j$. Differentiating the factorization in \eqref{eq:BKPSformula} with respect to $z_2$ gives
\[ \begin{aligned} 
\tfrac{\partial }{\partial z_2}\big[\bar{q}(z)-\alpha q(z)\big]&=(1-\alpha)\frac{\partial u^{\alpha}}{\partial z_2}(z )\prod_{k=1}^J\prod_{m=1}^{M_k}\left(z_2+q_k(z_1)+z_1^{2L_k}\psi^{\alpha}_{k,m}(z_1)\right)\\
&+
(1-\alpha)u^{\alpha}(z)\sum_{\#\textrm{ of factors in}\, \eqref{eq:BKPSformula} }\left( \prod_{\mathrm{one\,\, factor\,\, deleted}}\left(z_2+q_k(z_1)+z_1^{2L_k}\psi^{\alpha}_{k,m}(z_1)\right)\right).
\label{eq:BKPSder}
\end{aligned} \]
Recall that 
\[ h^{\alpha}_j(z_1) = -q_{j_0}(z_1)-z_1^{2L_{j_0}}\psi^{\alpha}_{j_0,m_0}(z_1), \]
for some pair $(j_0,m_0)$.
Then substituting $(x, h^{\alpha}_j(x))$ into the above $z_2$-derivative gives the following formula for the denominator of  $V^{\alpha}_j$
\[
\tfrac{\partial }{\partial z_2}\big[\bar{q}- \alpha q\big](x,h_j^{\alpha}(x)) = (1-\alpha)u^{\alpha}(x,h_j^{\alpha}(x)) \prod_{(k,m) \ne (j_0,m_0)}\Big(h^{\alpha}_j(x)+q_k(x)+ x^{2L_k}\psi^{\alpha}_{k,m}(x)\Big),
\]
where all but one term vanished when we substituted in $z_1=x$ and $z_2 = h^{\alpha}_j(x).$
Let $N_\alpha(h^\alpha_j, k, m)$ be the order of vanishing of the term
\[ h^{\alpha}_j(x)+q_k(x)+ x^{2L_k}\psi^{\alpha}_{k,m}(x)\]
at $x=0$, so that the order of vanishing of the denominator of $V^\alpha_j$ at $x=0$ is 
\[ \sum_{(k,m) \ne (j_0,m_0)} N_\alpha(h^\alpha_j, k, m).\]
We can similarly study the numerator of $V^\alpha_j$. Specifically, substituting $(x, h^{\alpha}_j(x))$ into the factorization of $q$ from \eqref{eqn:qfactor} gives
\[ q(x, h^{\alpha}_j(x)) = u(x, h^{\alpha}_j(x)) \prod_{k=1}^J\prod_{m=1}^{M_k}\left(h^{\alpha}_j(x)+q_k(x)+x^{2L_k}\psi_k(\mu^{m}_kx^{\frac{1}{M_k}})\right).\]
Let $N(h^\alpha_j, k, m)$ be the order of vanishing of the term
\[ h^{\alpha}_j(x)+q_k(x)+x^{2L_k}\psi_k(\mu^{m}_kx^{\frac{1}{M_k}})\]
at $x=0$, so that the order of vanishing of the numerator of $V^\alpha_j$ at $x=0$ is 
\[ \sum_{(k,m)} N(h^\alpha_j, k, m).\]
Because $\text{Im} (\psi_{j_0}(0) ) \ne 0,$ one can check that $N(h^\alpha_j, j_0,m_0) = K_j$. Furthermore, we claim that for each $(k,m) \ne (j_0,m_0)$ we have 
\begin{equation} \label{eqn:nalpha} N_\alpha(h^\alpha_j, k, m) = N(h^\alpha_j, k, m).\end{equation}
Once we have \eqref{eqn:nalpha}, comparing the numerator and denominator of $V^\alpha_j$ near $x=0$ will yield \eqref{eqn:valphabd}. 

We establish \eqref{eqn:nalpha} by contradiction: assume there is some $(k,m) \ne (j_0,m_0)$ such that $N_\alpha(h^\alpha_j, k, m) \ne N(h^\alpha_j, k, m).$ If either $N_\alpha(h^\alpha_j, k, m)$ or  $N(h^\alpha_j, k, m)$ was less than $2L_k$, they would have to be equal, since the underlying branches are equal to that order. So, it must be the case that one is greater than $2L_k$. As  $\text{Im}(\psi_{j_0}(0) ) \ne 0$, we must have $N(h^\alpha_j, k, m)\le 2L_k$ and so we can conclude that $ N_\alpha(h^\alpha_j, k, m) > 2L_k$.

This implies that both the $(m_0,j_0)$ and $(k,m)$ branches of $\mathcal{V}_\alpha$ are of the form
\begin{equation} \label{eqn:repeatbranch} x_2 = -q_k(x)- x^{2L_k}\psi^{\alpha}_{k,m}(0) +O\big(|x|^{2L_k +1}\big).\end{equation}
Fix $x$ close to zero and define $\Psi_x(x_2) := \Psi(x, x_2)$. Then 
\[\Psi_x (x_2) = \phi(\beta^{-1}(x), \beta^{-1}(x_2)).\]
Recall that $\phi(\zeta, \cdot)$ is a nonconstant finite Blaschke product for all but finitely many $\zeta \in \mathbb{T}$. Thus, by properties of finite Blaschke products,  for almost every $x$, as inputs to $\Psi_x$ go through the $x_2$ values between those of the two branches of $\mathcal{V}_\alpha$ of form \eqref{eqn:repeatbranch}, it must output each $\lambda \in \mathbb{T}$ at least once. This implies for each $\lambda \in \mathbb{T}_q$, there is actually a branch of $\mathcal{V}_\lambda$ of form \eqref{eqn:repeatbranch}. Then the discussion in the proof of Theorem 2.21 in  \cite{BKPS} implies that
 \[q_k(x)+x^{2L_k}\psi^{\alpha}_{k,m}(0)\]
 is an initial segment of a branch of $q$ of order $2 L_k +1$. As this is the exact condition discussed in Lemma \ref{lem:badalpha}, this contradicts the fact that we already removed such $\alpha$ values from consideration. This establishes \eqref{eqn:nalpha} and completes the proof.
\end{proof}

\begin{remark}
Note that it is indeed possible to have lower order of vanishing for certain values of $\alpha$, so that it is necessary to allow exclusion of some finite collection in the statement of Theorem \ref{thm:Waj}. See for instance \cite[Example 5.2]{BCSMich}, where all weights $W^{\alpha}$ exhibit order $4$ vanishing at the unique singularity of that RIF, except for $W^{-1}$ which vanishes to order 2.

\end{remark}

\section{A tridisk example}\label{sec:example}
For $s\geq 3$, consider the three-variable rational inner function
\begin{equation}
\phi_{s}(z)=\frac{\tilde{p}_s(z)}{p_s(z)}=\frac{sz_1z_2z_3-z_1z_2-z_1z_3-z_2z_3}{s-z_1-z_2-z_3}, \quad z\in \mathbb{D}^3.
\label{3dfave}
\end{equation}
This function and its close relatives often appear as basic tridisk examples, see e.g. \cite{Kne11, BKPS}. When $s>3$, the polynomial $p_s$ has no zeros in the closed tridisk. Hence $\phi_s$ has no singularities on $\overline{\mathbb{D}^3}$, and Theorem \ref{thm:dclarkformula} applies.
A computation shows that
\[\frac{\partial \phi_s}{\partial z_3}(z)=\frac{s^2z_1z_2-s(z_1^2z_2+z_1z_2^2+z_1+z_2)+z_1^2+z_1z_2+z_2^2}{(s-z_1-z_2-z_3)^2}.\]
For $\alpha\in \mathbb{T}$ fixed, the set $\{\zeta \in \mathbb{T}^3\colon \tilde{p}_s(\zeta)-\alpha p_s(\zeta)=0\}$ can be parametrized as
\[\zeta_3=\psi^{\alpha}_s(\zeta_1,\zeta_2)=\frac{\alpha s-\alpha \zeta_1-\alpha \zeta_2+\zeta_1\zeta_2}{s\zeta_1\zeta_2-\zeta_1-\zeta_2+\alpha}, \quad (\zeta_1, \zeta_2) \in \mathbb{T}^2.\]
As is guaranteed by \cite[Theorem 4.8]{BPSmulti}, each $\psi^{\alpha}_s$ is the reciprocal of an RIF on $\mathbb{D}^2$, and each $\psi^{\alpha}_s$ is continuous on $\overline{\mathbb{D}^2}$ when $s>3$ since $\phi_s$ has no singularities. This can also be checked directly in this simple case. Plugging $\zeta_3=\psi^{\alpha}_s$ into $\frac{\partial \phi_s}{\partial z_3}$ and simplifying, we get that for $\alpha\in \mathbb{T}$ fixed and for $f\in C(\mathbb{T}^3)$, the Clark measure $\sigma_{s, \alpha}$ satisfies
\[\int_{\mathbb{T}^3}f(\zeta)d\sigma_{s,\alpha}(\zeta)=\int_{\mathbb{T}^2}f\left(\zeta_1,\zeta_2, \psi^{\alpha}_s(\zeta_1,\zeta_2)\right)W_{s,\alpha}(\zeta_1,\zeta_2)dm(\zeta_1,\zeta_2),\]
where 
\[W_{s,\alpha}(\zeta_1,\zeta_2)=\left|\frac{s^2\zeta_1\zeta_2-s(\zeta_1^2\zeta_2+\zeta_1\zeta_2^2+\zeta_1+\zeta_2)+\zeta_1^2+\zeta_1\zeta_2+\zeta_2^2}{(s\zeta_1\zeta_2-\zeta_1-\zeta_2+\alpha)^2}\right|.\]
Thus, when $s>3$, each weight $W_{s,\alpha}$ is continuous and bounded above and below on $\mathbb{T}^2$, but even for this simple choice of RIF, the explicit representation of  the $\sigma_{s, \alpha}$ involves some fairly complicated expressions.

\begin{figure}[h!]
    \subfigure[Level set for $\phi=(3z_1z_2z_3-z_1z_2-z_1z_3-z_2z_3)/(3-z_1-z_2-z_3)$ for $\alpha=i$.]
      {\includegraphics[width=0.4 \textwidth]{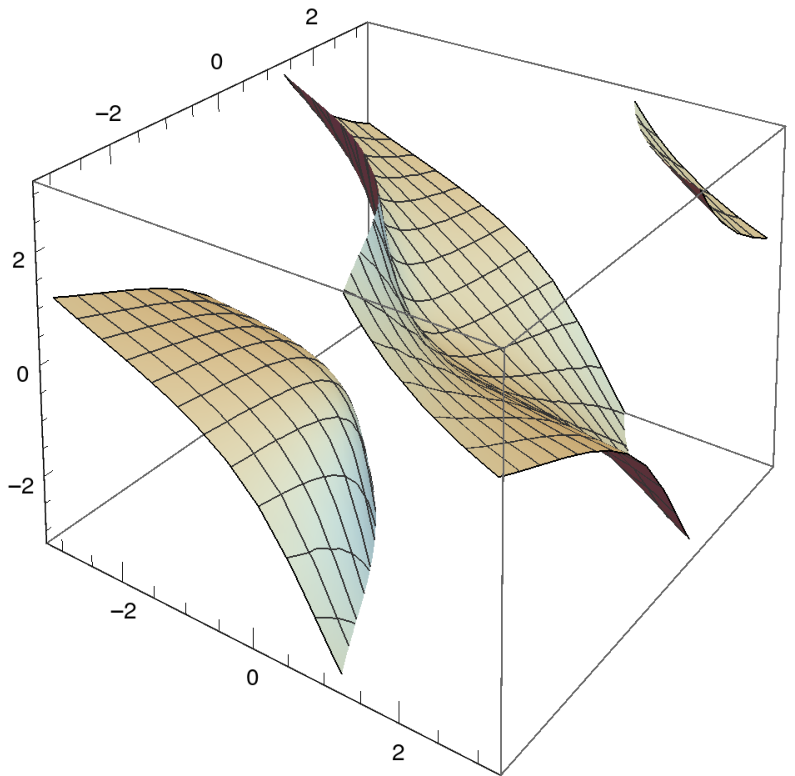}}
    \hspace{.5in}
    \subfigure[Level set for $\phi=(3z_1z_2z_3-z_1z_2-z_1z_3-z_2z_3)/(3-z_1-z_2-z_3)$ for $\alpha=1$ and $\alpha=-1$ (salmon).]
      {\includegraphics[width=0.4 \textwidth]{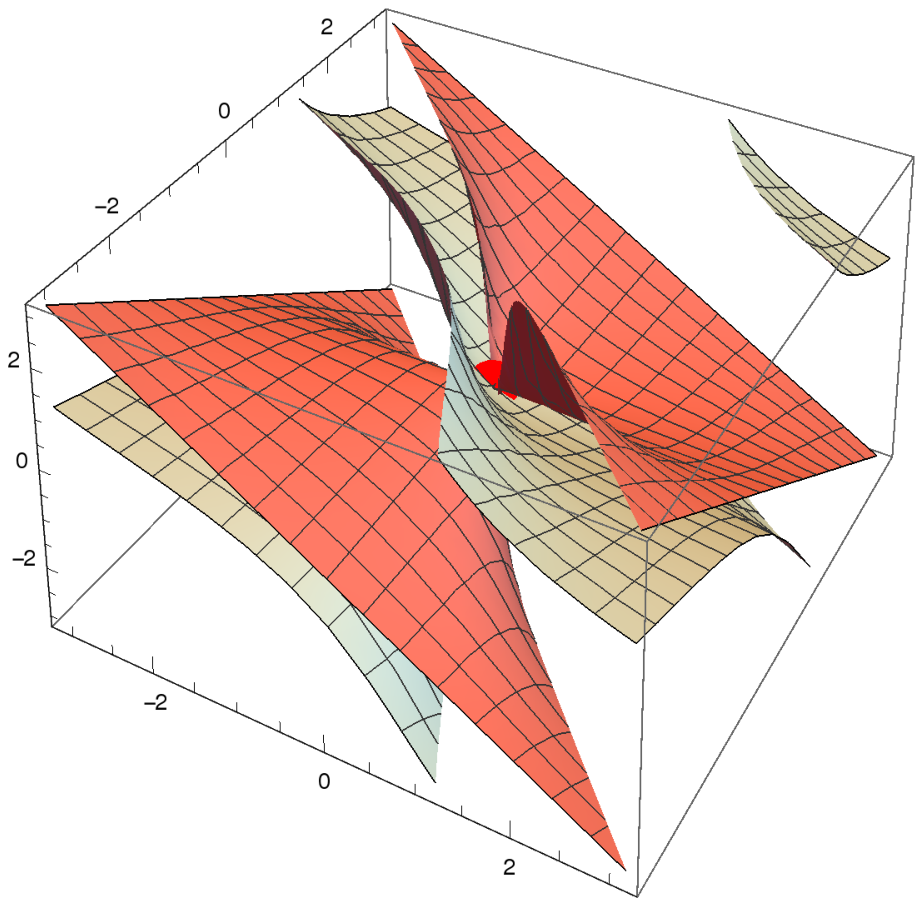}}
  \caption{\textsl{Supports of the Clark measures $\sigma_{3, \alpha}$ for $\phi_3$.}}
   \label{fig:3var}
\end{figure}

We now turn to the critical case $s=3$. Then $\phi_3$ has a singularity at $(1,1,1)\in \mathbb{T}^3$, with $\phi^*_3(1,1,1)=\angle \lim_{z\to (1,1,1)}\phi_3(z)=-1$, and we check that 
$\phi_3(1,1,z_3)\equiv -1$, reflecting the fact that the corresponding Blaschke factor experiences a degree drop. For all $\alpha \neq -1$, the two-variable RIF $1/\psi^{\alpha}_3$ is continuous on $\overline{\mathbb{D}^2}$, and the weight $W_{3,\alpha}$ also remains continuous. However, for each $\alpha \in \mathbb{T}\setminus \{-1\}$, we have $W_{3,\alpha}(1,1)=0$. Finally, examining what happens for $\alpha=-1$, the non-tangential value of $\phi_3$ at its singularity, reveals some of the difficulties that can arise in higher dimensions. First of all,
\[\zeta_3 = \psi^{-1}_3(\zeta_1,\zeta_2)=\frac{-3+\zeta_1+\zeta_2+\zeta_1\zeta_2}{-1-\zeta_1-\zeta_1+3\zeta_1\zeta_2}\]
is the reciprocal of an RIF with a singularity at $(1,1)$, illustrating the fact that the level set $\mathcal{C}_{-1}$ cannot be viewed as a smooth surface in the three-torus. Figure \ref{fig:3var} shows the graphs of $\mathcal{C}_i$, $\mathcal{C}_1$, and $\mathcal{C}_{-1}$ on $\mathbb{T}^3$, where points are associated with their arguments in $[-\pi, \pi)^3$.

Moreover,  we see that 
\[W^*(\zeta_1,\zeta_2)=\lim_{\alpha\to -1}W_{3,\alpha}(\zeta_1,\zeta_2)=\left|\frac{-3(\zeta_1^2\zeta_2+\zeta_1\zeta_2^2+\zeta_1+\zeta_2)+\zeta_1^2+10\zeta_1\zeta_2+\zeta_2^2}{(3\zeta_1\zeta_2-\zeta_1-\zeta_2-1)^2}\right|\]
is a discontinuous function on $\mathbb{T}^2$, which is moreover unbounded near $(1,1)\in \mathbb{T}^2$, as can be verified by evaluating along the curve $\{(e^{i\theta}, e^{-i\theta})\}\subset \mathbb{T}^2$ to obtain the expression $W^*(e^{i\theta}, e^{-i\theta})=1+\frac{1}{1-\cos \theta}$. 

Given this example, it would appear that a more sophisticated approach is needed to handle Clark measures for RIFs in higher dimensions that possess singularities.

\section*{Acknowledgements}
Bickel was partially supported  by National Science Foundation DMS grant \#2000088. Sola was partially supported by the National Science Foundation under DMS grant \#1928930 while he participated in a program hosted by MSRI (Berkeley, CA) during the spring 2022 semester.

\end{document}